\documentclass[11pt]{article}

\usepackage{amsmath}
\usepackage{amsthm}
\usepackage{amsfonts}
\usepackage{bbm}
\usepackage{amssymb}
\usepackage{color}
\usepackage{xcolor}
\usepackage{hyperref}

\newcommand{\me}{\mathbb{E}}
\newcommand{\mr}{\mathbb{R}}

\newcommand{\mn}{\mathbb{N}}
\newcommand{\mmp}{\mathbb{P}}

\newcommand{\eee}{{\rm e}}
\DeclareMathOperator{\1}{\mathbbm{1}}

\newtheorem{thm}{Theorem}[section]
\newtheorem{lemma}[thm]{Lemma}

\newtheorem{cor}[thm]{Corollary}

\newtheorem{assertion}[thm]{Proposition}
\theoremstyle{definition}

\theoremstyle{remark}
\newtheorem{rem}[thm]{Remark}

\begin{document}

\title{Limit theorems for discounted convergent perpetuities II}\date{}
\author{Alexander Iksanov\footnote{Faculty of Computer Science and Cybernetics, Taras Shevchenko National University of Kyiv, Ukraine; e-mail address:
iksan@univ.kiev.ua} \ \ Alexander Marynych\footnote{Faculty of Computer Science and Cybernetics, Taras Shevchenko National University of Kyiv, Ukraine; e-mail address:
marynych@unicyb.kiev.ua} \ \ Anatolii Nikitin\footnote{Faculty of Natural Sciences, Jan Kochanowski University of Kielce, Poland and Faculty of Economics, National University of Ostroh Academy, Ukraine; e-mail address: anatolii.nikitin@ujk.edu.pl}
}
\maketitle
\begin{abstract}
\noindent Let $(\xi_1, \eta_1)$, $(\xi_2, \eta_2),\ldots$ be independent identically distributed $\mathbb{R}^2$-valued random vectors. Assuming that $\xi_1$ has zero mean and finite variance and imposing three distinct groups of assumptions on the distribution of $\eta_1$ we prove three functional limit theorems for the logarithm of convergent discounted perpetuities $\sum_{k\geq 0}\eee^{\xi_1+\ldots+\xi_k-ak}\eta_{k+1}$ as $a\to 0+$. Also, we prove a law of the iterated logarithm which corresponds to one of the aforementioned functional limit theorems. The present paper continues a line of research initiated in the paper Iksanov, Nikitin and Samoillenko (2022), which focused on limit theorems for a different type of convergent discounted perpetuities.
\end{abstract}


\hphantom{2000 Mathematics Subject Classification: } Secondary: 60G50, 60G55

\section{Introduction}

Let $(\xi_1, \eta_1)$, $(\xi_2,\eta_2),\ldots$ be independent copies of an $\mr^2$-valued random vector $(\xi,\eta)$ with  arbitrarily dependent components. Denote by $(S_k)_{k\in\mn_0}$ (as usual, $\mn_0:=\mn\cup\{0\}$) the standard random walk with jumps $\xi_k$ defined by $S_0:=0$ and $S_k:=\xi_1+\ldots+\xi_k$ for $k\in\mn$. Whenever a random series $\sum_{k\geq 0}\eee^{S_k}\eta_{k+1}$ converges almost surely (a.s.), its sum is called {\it perpetuity} because of the following financial application. Assuming temporarily that $\eta$ is a.s.\ positive, the variables $\eta_{k+1}$ and $\eee^{\xi_k}$ may be interpreted as the planned payoff of a private pension fund to a client and the discount factor for year $k\in\mn_0$, respectively. The pension payoffs to a client are made at the beginning of each year. The variable $\sum_{k\geq 0}\eee^{S_k}\eta_{k+1}$ can be thought of as a perpetuity, that is, the present value of a permanent commitment to make a payoff annually into the future forever. In other words, this is the initial payment to a fund and, for $k\in\mn_0$, $\eee^{S_k}\eta_{k+1}$ is an amount ensuring that a client gets the planned payoff $\eta_{k+1}$ at the beginning of year $k\in\mn_0$.

It is known (see Lemma 1.7 in \cite{Vervaat:1979} or Theorem 2.1 in \cite{Goldie+Maller:2000}) that the conditions $\me \xi\in [-\infty,0)$ and $\me \log^+ |\eta|<\infty$ ensure that the series $\sum_{k\geq 0}\eee^{S_k}\eta_{k+1}$ (absolutely) converges a.s. Recall that $\log^+x=\log x$ if $x\geq 1$ and $=0$ if $x\in (0,1)$. Further detailed information on perpetuities, accumulated up to 2016, can be found in the books \cite{Buraczewski et al:2016} and \cite{Iksanov:2016}.

\subsection{Previously investigated discounted perpetuities}

There are several ways to define a {\it discounted convergent perpetuity}. One option is
$$
X(b):=\sum_{k\geq 0}b^{S_k}\eta_{k+1},\quad b\in (0,1)
$$
or equivalently $\sum_{k\geq 0}\eee^{-S_k/t}\eta_{k+1}$ for $t>0$. In the recent article \cite{Iksanov+Nikitin+Samoilenko:2021} basic limit theorems for $X(b)$, properly normalized, as $b\to 1-$, were proved, namely, a strong law of large numbers, a functional central limit theorem and a law of the iterated logarithm. To be more specific, we state a combination of Theorem 1.2 and one part of Theorem 1.5 in \cite{Iksanov+Nikitin+Samoilenko:2021} as Proposition \ref{reminder}. Denote by $C=C(0,\infty)$ the space of continuous functions defined on $(0,\infty)$ equipped with the locally uniform topology. Throughout the paper we write $\Longrightarrow$ to denote weak convergence of probability measures in a function space.
\begin{assertion}\label{reminder}
Assume that $\mu=\me\xi\in (0,\infty)$, $\me \eta=0$ and ${\tt s}^2:={\rm Var}\,\eta \in (0,\infty)$. Then, as $b\to 1-$,
\begin{equation}\label{clt22}
\Big((1-b^2)^{1/2}\sum_{k\geq 0}b^{uS_k}\eta_{k+1}\Big)_{u>0} ~\Longrightarrow~ (2{\tt s}^2\mu^{-1})^{1/2}\Bigg(\int_{[0,\,\infty)}e^{-uy}{\rm d}B(y)\Bigg)_{u>0}
\end{equation}
on $C$, where $(B(t))_{t\geq 0}$ is a standard Brownian motion, and
\begin{equation}\label{auxlimsup}
{\lim\sup\,(\lim
\inf)}_{b\to 1-}\Big(\frac{1-b^2}{\log\log \frac{1}{1-b^2}}\Big)^{1/2}\sum_{k\geq 0}b^{S_k}\eta_{k+1}=+(-)(2{\tt s}^2 \mu^{-1})^{1/2}\quad\text{{\rm a.s.}}
\end{equation}
\end{assertion}

Note that in the cited Theorem 1.2 weak convergence was stated on the Skorokhod space $D(0,\infty)$ of c\`{a}dl\`{a}g functions on $(0,\infty)$ equipped with the $J_1$-topology. Since the process on the left-hand side of \eqref{clt22} is a.s.\ continuous in $u$, the weak convergence also takes place on $C$. Corollary 1.5 in \cite{Iksanov+Kondratenko:2021} is an ultimate version of the functional central limit theorem for $X(b)$ in the case ${\rm Var}\,\xi<\infty$ and ${\tt s}^2<\infty$, which particularly strengthens \eqref{clt22}. In \cite{Iksanov+Kondratenko:2021}, the condition $\me \eta\in\mr$ is allowed, which is a new aspect in comparison to \eqref{clt22}.

In both \eqref{clt22} and \eqref{auxlimsup}, the random walk $(S_k)$ only provides a first-order contribution to the limit which is represented by the strong law of large numbers $\lim_{k\to\infty} k^{-1}S_k=\mu$ a.s. In other words, the limits remain unchanged on replacing $S_k$ on the left-hand sides with $\mu k$, see Theorem 1.1 in \cite{Bovier+Picco:1993} for the corresponding counterpart of \eqref{auxlimsup}. Limit relations \eqref{clt22} and \eqref{auxlimsup} are mainly driven by fluctuations of the random walk $(\eta_1+\ldots+\eta_n)$ as $n$ becomes large. More precisely, the main driving forces behind \eqref{clt22} and \eqref{auxlimsup} are the Donsker functional limit theorem and the law of the iterated logarithm for $(\eta_1+\ldots+\eta_n)$, respectively.

\subsection{New type of discounted perpetuities and main results}\label{sec:newtype}

Our standing assumptions throughout the paper are: $\eta$ is a.s.\ positive, $\me \log^+\eta<\infty$ and
\begin{equation}\label{10}
\me \xi=0\quad \text{and}\quad \sigma^2:={\rm Var}\,\xi\in (0,\infty).
\end{equation}
We shall investigate
$$
Y(a):=\sum_{k\geq 0}\eee^{S_k-ak}\eta_{k+1},\quad a>0,
$$
which is yet another type of discounted convergent perpetuity, and an accompanying process
$$
Z(a):=\sup_{k\geq 0}\,(S_k-ak+\log \eta_{k+1}),\quad a>0.
$$
By Theorem 2.1 in \cite{Goldie+Maller:2000}, the latter series converges a.s., that is, the perpetuity is indeed convergent. This implies that $\lim_{k\to\infty}(S_k-ak+\log \eta_{k+1})=-\infty$ a.s., whence $|Z(a)|<\infty$ a.s.\ for each $a>0$. Specifically, we shall prove functional limit theorems for $(\log Y(au))_{u>0}$, properly normalized, as $a\to 0+$ and a law of the iterated logarithm for $\log Y(a)$, again properly normalized. Note that $\lim_{a\to 0+} \sum_{k\geq 0}\eee^{S_k-ak}\eta_{k+1}=\sum_{k\geq 0}\eee^{S_k}\eta_{k+1}=+\infty$ a.s. Here, noting that our assumptions entail $\mmp\{\eta+\eee^\xi c =c\}<1$ for all $c\in\mr$, the a.s.\ divergence is justified by Theorem 2.1 in \cite{Goldie+Maller:2000}. Thus, some normalizations are indeed needed in our limit theorems.

The presence of the logarithm already shows that limit theorems for $X(b)$ and $Y(a)$ are of different nature. It will follow from our proofs that the functional limit theorems for $Y(a)$ are driven by heavy-traffic limit theorems for $Z(a)$ as $a\to 0+$. The sequence $(S_k-ak+\log \eta_{k+1})_{k\in\mn_0}$ is a globally perturbed random walk, see \cite{Iksanov:2016} for a survey. The asymptotics of its supremum depends heavily upon the interplay between the asymptotic growth of $(S_k-ak)_{k\in\mn_0}$ and that of $(\log\eta_j)_{j\in\mn}$. This fact leads to three different functional limit theorems stated in Theorems \ref{perturbed_weak0}, \ref{perturbed_weak} and \ref{main_perturbed_weak convergence}. We note in passing that the one-dimensional distributional convergence of $\sup_{k\geq 0}(S_k-ak)$ as $a\to 0+$, properly normalized, is well-understood for the random walks $(S_k)_{k\in\mn_0}$ attracted to a centered stable L\'{e}vy process, see \cite{Shneer+Wachtel:2011} and references therein. Also, we mention that the one-dimensional distributional convergence of $\sup_{k\geq 0}(\log \eta_{k+1}-ak)$, properly normalized, as $a\to 0$ was investigated in \cite{Daley+Hall:1984}, see, in particular, Theorem 7 therein.

Similarly, the law of the iterated logarithm for $\log Y(a)$ stated in Theorem \ref{thm:lil} is a consequence of the law of the iterated logarithm for $\max_{0\leq k\leq n}\,(S_k+\log \eta_{k+1})$, properly normalized, as $n\to\infty$ and a previously known deterministic continuity result recalled in Proposition \ref{asser:det}.

Under the aforementioned financial interpretation, the variable $Y(a)$ is a perpetuity with a discount factor for year $k$ being equal to $\eee^{\xi_k-a}$. It is natural to call $-(\me\xi-a)=a$ the average rate of exponential wealth growth in the economics. Thus, our limit theorems describe the fluctuations of the perpetuity, when the average rate of exponential wealth growth approaches $0$ while staying positive.

We are ready to formulate our main results.

\subsubsection{Weak convergence} According to Lemma \ref{lem:cont}, the processes
$$
\big(\sup_{k\geq 0}\,(S_k-auk+\log \eta_{k+1})\big)_{u>0}\quad\text{and}\quad\Big(\log \sum_{k\geq 0}\eee^{S_k-auk}\eta_{k+1}\Big)_{u>0}
$$
are a.s.\ continuous. This enables us to formulate functional limit theorems in $C$.

We start with simpler situations in which the asymptotic behavior of the discounted convergent perpetuity is driven by either fluctuations of $(S_k-ak)_{k\in\mn_0}$ (Theorem \ref{perturbed_weak0}) or $(\log \eta_j)_{j\in\mn}$ (Theorem \ref{perturbed_weak}).
\begin{thm}\label{perturbed_weak0}
Suppose that \eqref{10} holds and that
\begin{equation}\label{cond1}\lim_{t\to\infty}t^2\mmp\{\log \eta>t\}=0.
\end{equation}
Then
\begin{equation}\label{eq:supr}
\big(a\,\sup_{k\geq 0}\,(S_k-auk+\log \eta_{k+1})\big)_{u>0} \quad \Longrightarrow\quad \big(\sup_{s\geq 0}\,(\sigma B(s)-us)\big)_{u>0},\quad
a\to 0+
\end{equation}
and
\begin{equation}\label{limit_perturbed0}
\Big(a\log \sum_{k\geq 0}\eee^{S_k-auk}\eta_{k+1}\Big)_{u>0} \quad \Longrightarrow\quad \big(\sup_{s\geq 0}\,(\sigma B(s)-us)\big)_{u>0},\quad
a\to 0+
\end{equation}
on $C$, where $B$ is a standard Brownian motion.
\end{thm}
\begin{rem}
The limit process in Theorem \ref{perturbed_weak0} is the Legendre-Fenchel transform of $s\mapsto -\sigma B(s)\1_{[0,\,\infty)}(s)$, $s\in\mathbb{R}$ evaluated at $-u<0$.
In particular, it is a.s.\ convex (as a function of $u$), hence a.s.\ continuous. Similarly, the converging process in \eqref{eq:supr} can be thought of as a discrete version of the Legendre-Fenchel transform. These observations are implicitly used in the proof of Lemma \ref{lem:cont} below when showing the a.s.\ convexity of the processes involved.
\end{rem}

For positive $\gamma$ and $\rho$, let $N^{(\gamma,\,\rho)}:=\sum_k \varepsilon_{(t_k^{(\gamma,\,\rho)},\,j_k^{(\gamma,\,\rho)})}$ be a Poisson random
measure on $[0,\infty)\times (0,\infty]$ with intensity measure $\mathbb{LEB}\times \mu_{\gamma,\,\rho}$, where $\varepsilon_{(t,\,x)}$ is the probability measure
concentrated at $(t,x)\in [0,\infty)\times (0,\infty]$, $\mathbb{LEB}$ is the Lebesgue measure on $[0,\infty)$, and $\mu_{\gamma,\,\rho}$ is the measure on $(0,\infty]$ defined by
$$\mu_{\gamma,\,\rho}\big((x,\infty]\big)=\gamma x^{-\rho},\quad x>0.$$
\begin{thm}\label{perturbed_weak}
Suppose that \eqref{10} holds and that
the function $t\mapsto \mmp\{\log \eta>t\}$ is regularly varying at $\infty$ of index $-\beta$, $\beta\in (1, 2]$. If $\beta=2$, assume additionally that $\lim_{t\to\infty} t^2\mmp\{\log \eta>t\}=\infty$.
Let $b$ and $c$ be positive functions which satisfy $\lim_{t\to\infty} t\mmp\{\log \eta>b(t)\}=1$ and $b(c(a))\sim ac(a)$ as $a\to 0+$. Then
\begin{equation}\label{limit_perturbed1111}
\Big(\frac{1}{ac(a)}\,\sup_{k\geq 0}\,(S_k-auk+\log \eta_{k+1})\Big)_{u>0}\quad\Longrightarrow\quad
\big(\sup_k\,\big(-ut_k^{(1,\,\beta)}+j_k^{(1,\,\beta)}\big)\big)_{u>0},\quad
a\to 0+
\end{equation}
and
\begin{equation}\label{limit_perturbed1}
\Big(\frac{1}{ac(a)}\log \sum_{k\geq 0}\eee^{S_k-auk}\eta_{k+1}\Big)_{u>0}\quad\Longrightarrow\quad
\big(\sup_k\,\big(-ut_k^{(1,\,\beta)}+j_k^{(1,\,\beta)}\big)\big)_{u>0},\quad
a\to 0+
\end{equation}
on $C$.
\end{thm}
\begin{rem}\label{rem1}
In the role of $b$ one can take an asymptotically inverse function of $t\mapsto 1/\mmp\{\log \eta>t\}$. By Theorem 1.5.12 in \cite{Bingham+Goldie+Teugels:1989}, such functions exist and are regularly varying at $\infty$ of index $1/\beta$. In the role of $t\mapsto c(1/t)$ one can take an asymptotically inverse function of $t\mapsto t/b(t)$. Another appeal to Theorem 1.5.12 in \cite{Bingham+Goldie+Teugels:1989} enables us to conclude that $t\mapsto c(1/t)$ is regularly varying at $\infty$ of index $\beta/(\beta-1)$. Hence, $a\mapsto c(a)$ is regularly varying at $0+$ of index $-\beta/(\beta-1)$. In particular, if $\mmp\{\log \eta>t\}\sim \kappa t^{-\beta}$ as $t\to\infty$ for some $\kappa>0$, then
$c(a)\sim \kappa^{1/(\beta-1)}a^{-\beta/(\beta-1)}$ as $a\to 0+$. For later needs, we note that
\begin{equation}\label{eq:a2c(a)}
\lim_{a\to 0+}a^2c(a)=\infty.
\end{equation}
This is obvious when $\beta\in (1,2)$ and follows from 
$$
a^2c(a)\sim a^2c^2(a)\mmp\{\log\eta>b(c(a))\}\sim (ac(a))^2\mmp\{\log\eta>ac(a)\}\to \infty ,\quad a\to 0+
$$
when $\beta=2$.

According to formula \eqref{38888}, for each $u>0$, $\sup_k\,\big(-ut_k^{(1,\,1)}+j_k^{(1,\,1)}\big)=+\infty$ a.s. This explains the fact that Theorem \ref{perturbed_weak} is not applicable in the situations in which $\me \log^+ \eta<\infty$ and $t\mapsto \mmp\{\log \eta>t\}$ is regularly varying at $\infty$ of index $-1$.
\end{rem}
\begin{rem}
Observe that, under the assumptions of Theorem \ref{perturbed_weak0}, both \eqref{eq:supr} and \eqref{limit_perturbed0} remain true on replacing $\eta$ with $1$ and that, under the assumptions of Theorem \ref{perturbed_weak}, both \eqref{limit_perturbed1111} and \eqref{limit_perturbed1} remain true on replacing $\xi$ with $0$. We think this (obvious) observation facilitates understanding of Theorems \ref{perturbed_weak0} and \ref{perturbed_weak}.
\end{rem}

If in addition to \eqref{10} the condition
\begin{equation}\label{11}
\mmp\{\log \eta>t\}~\sim~ \lambda t^{-2},\quad t\to\infty,
\end{equation}
holds for some $\lambda>0$, then contributions of $\max_{0\leq k\leq
n} S_k$ and $\max_{1\leq k\leq n+1}\log \eta_k$ to the asymptotic behavior of $\max_{0\leq k\leq
n}(S_k+\log \eta_{k+1})$ are comparable. This situation which
is more interesting than the other two is treated in Theorem
\ref{main_perturbed_weak convergence}.
\begin{thm}\label{main_perturbed_weak convergence}
Suppose that \eqref{10} and \eqref{11} hold. Then
\begin{equation}\label{limit_suprem}
\big(a \sup_{k\geq 0}\,(S_k-auk+\log \eta_{k+1})\big)_{u>0} \quad \Longrightarrow\quad
\big(\sup_k \big(\sigma B(t_k^{(\lambda,\,2)})-u t_k^{(\lambda,\,2)}+j_k^{(\lambda,\,2)}\big)\big)_{u>0},\quad a\to 0+
\end{equation}
and
\begin{equation}\label{limit_perturbed}
\Big(a\log \sum_{k\geq 0}\eee^{S_k-auk}\eta_{k+1}\Big)_{u>0} \quad \Longrightarrow\quad
\big(\sup_k \big(\sigma B(t_k^{(\lambda,\,2)})-u t_k^{(\lambda,\,2)}+j_k^{(\lambda,\,2)}\big)\big)_{u>0},\quad a\to 0+
\end{equation}
on $C$, where $B$ is a standard Brownian motion independent of $N^{(\lambda,\,2)}$.
\end{thm}

\subsubsection{A law of the iterated logarithm} Theorem \ref{thm:lil}(a) is a law of the iterated logarithm for $\log Y(a)$ which corresponds to the distributional convergence of Theorem \ref{perturbed_weak0}. Theorem \ref{thm:lil}(b) is a law of the iterated logarithm for closely related random variables $\log \int_0^\infty \eee^{B(s)-as}{\rm d}s$, where $a>0$ and $B$ is a standard Brownian motion. The variable $\int_0^\infty \eee^{B(s)-as}{\rm d}s$ which is known in the literature as an {\it exponential functional of Brownian motion} has been the object of intensive research in the recent past, see \cite{Yor:2001} for a collection of results in a book format. According to formula \eqref{eq:convtoexp} below, $a\log\int_0^\infty \eee^{B(s)-as}{\rm d}s$ converges in distribution as $a\to 0+$ to an exponentially distributed random variable. This serves an informal explanation of the fact that one factor of the normalization in Theorem \ref{thm:lil} is $\log\log 1/a$ rather than $(\log\log 1/a)^{1/2}$ which typically arises in the cases when the limit distribution is normal.
\begin{thm}\label{thm:lil}
(a) Put $f(x):=x^2/\log\log x$ for $x>\eee$. Suppose that \eqref{10} holds and that
\begin{equation}\label{eq:oneta}
\me f(\log^+\eta)<\infty.
\end{equation}
Then
\begin{equation}\label{eq:perpLIL}
{\lim\sup}_{a\to 0+}\frac{2a \log \sum_{k\geq 0}\eee^{S_k-ak}\eta_{k+1}}{\log\log (1/a)}=\sigma\quad {\rm a.s.}
\end{equation}

\noindent (b) Let $B$ be a standard Brownian motion. Then
\begin{equation}\label{eq:brm}
{\lim\sup}_{a\to 0+}\frac{2a \log\int_0^\infty \eee^{B(s)-as}{\rm d}s}{\log\log (1/a)}=1\quad {\rm a.s.}
\end{equation}
\end{thm}
\begin{rem}
There exist distributions of $\eta$ which satisfy \eqref{cond1} (the assumption of Theorem \ref{perturbed_weak0}) and do not satisfy \eqref{eq:oneta} (the assumption of Theorem \ref{thm:lil}). To exemplify, let $\mmp\{\log \eta>t\}\sim t^{-2}(\log t)^{-1}$ as $t\to\infty$. It will be explained in Remark \ref{rem2} that, under \eqref{10}, relation \eqref{eq:perpLIL} fails to hold for the aforementioned distributions of $\eta$.
\end{rem}
For a family of functions or a sequence $(x_t)$ denote by $C((x_t))$ the set of its limit points.
\begin{cor}\label{corr:lil}
Under the assumptions of Theorem \ref{thm:lil},
$$C\Big(\Big(\frac{2a \log \sum_{k\geq 0}\eee^{S_k-ak}\eta_{k+1}}{\log\log (1/a)}: a\in (0, 1/\eee)\Big)\Big)=[0,\sigma]\quad\text{{\rm a.s.}}$$ and
$$C\Big(\Big(\frac{2a \log\int_0^\infty \eee^{B(s)-as}{\rm d}s}{\log\log (1/a)}: a\in (0, 1/\eee)\Big)\Big)=[0,1]\quad\text{{\rm a.s.}}$$
\end{cor}

\section{Marginal limit distributions and continuity of the paths}\label{marginal}
This section commences with a short discussion of continuous-time counterparts of the discounted convergent perpetuities $Y(a)$. The advantage of the latter is availability of explicit formulae for their marginal distributions. This makes their analysis easier in comparison to that of $Y(a)$. The appearance below of the distributions of the suprema of certain L\'{e}vy processes with a drift in the role of limit distributions provides a hint towards what can be expected in the discrete setting.

In what follows, $\overset{{\rm d}}{=}$ and ${\overset{{\rm d}}\longrightarrow}$ denote equality of distributions and convergence in distribution, respectively. Let $\theta_{b,\,c}$ be a random variable having a gamma distribution with positive parameters $b$ and $c$, that is, $$\mmp\{\theta_{b,\,c} \in {\rm d}x\}=\frac{c^b x^{b-1}}{\Gamma(b)}\eee^{-cx}\1_{(0,\infty)}(x){\rm d}x,$$ where $\Gamma$ is the Euler gamma function. Note that $\theta_{1,\,c}$ is an exponentially distributed random variable of mean $1/c$.

Let $(B(s))_{s\geq 0}$ be a standard Brownian motion. A known result (Proposition 3 in \cite{Pollak+Siegmund:1986}, Proposition 4.4.4 (b) in \cite{Dufresne:1990}, Example 3.3 on p.~309 in \cite{Urbanik:1992}) states that, for each $a>0$,
\begin{equation}\label{Urban}
\int_0^\infty \eee^{B(s)-as}{\rm d}s~\overset{{\rm d}}{=}~ 2/\theta_{2a,\,1}.
\end{equation}
From this we infer
\begin{equation}\label{eq:convtoexp}
a\log\int_0^\infty \eee^{B(s)-as}{\rm d}s~{\overset{{\rm d}}\longrightarrow}~\theta_{1,\,2},\quad a\to 0+.
\end{equation}
The appearance of an exponential distribution may look mysterious, unless it is interpreted via the distributional equality $$\theta_{1,\,2} ~\overset{{\rm d}}{=}~ \sup_{s\geq 0}\,(B(s)-s),$$ which follows from Corollary 2 (ii) on p.~190 in \cite{Bertoin:1998}. More generally, let $X:=(X(s))_{s\geq 0}$ be a centered spectrally negative L\'{e}vy process. Then, by the same corollary in \cite{Bertoin:1998}, $\sup_{s\geq 0}\,(X(s)-s)~\overset{{\rm d}}{=}~\theta_{1,\,\tau}$, where $\tau>0$ is the largest solution to the equation $\eee^{-s}\me \eee^{sX(1)}=1$. If $X=B$ a Brownian motion, then the latter equation is equivalent to $s^2/2-s=0$, whence $\tau=2$. By the same reasoning, for each $u>0$ and each $w\in\mr$,
\begin{equation}\label{eq:expo}
\sup_{s\geq 0}\,(wB(s)-us)~\overset{{\rm d}}{=}~\theta_{1,\,2u/w^2}.
\end{equation}
Assume additionally that $X$ is an $\alpha$-stable L\'{e}vy process, $\alpha\in (1,2]$. Then arguing along the lines of the proof of Theorem \ref{perturbed_weak0} one can show that
$$a^{\alpha-1}\log\int_0^\infty \eee^{X(s)-as}{\rm d}s~{\overset{{\rm d}}\longrightarrow}~ \sup_{s\geq 0}\,(X(s)-s)~\overset{{\rm d}}{=}~\theta_{1,\,\tau},\quad a\to 0+.$$

According to \eqref{eq:expo}, the marginal limit distributions in Theorem \ref{perturbed_weak0} are exponential with means $\sigma^2/(2u)$. In Proposition \ref{prop:marginal} we identify the marginal limit distributions in Theorems \ref{perturbed_weak} and \ref{main_perturbed_weak convergence} and justify the claim made in Remark \ref{rem1}.
\begin{assertion}\label{prop:marginal}
Let $x,u,T>0$.

\noindent (a) For $\lambda>0$ and $\beta\in (1,2]$,
\begin{equation}\label{3889}
\mmp\Big\{\sup_{k:\, t_k^{(\lambda,\,\beta)}\leq T}\,\big(-u t_k^{(\lambda,\,\beta)}+j_k^{(\lambda,\,\beta)}\big)\leq x\Big\}=\exp(-u^{-1}(\beta-1)^{-1}\lambda (x^{1-\beta}-(x+uT)^{1-\beta}),
\end{equation}
\begin{equation}\label{388}
\mmp\Big\{\sup_k \big(-u t_k^{(\lambda,\,\beta)}+j_k^{(\lambda,\,\beta)}\big)\leq x\Big\}=\exp(-u^{-1}(\beta-1)^{-1}\lambda x^{1-\beta})
\end{equation}
and
\begin{equation}\label{38888}
\mmp\Big\{\sup_k \big(-u t_k^{(\lambda,\,1)}+j_k^{(\lambda,\,1)}\big)\leq x\Big\}=0.
\end{equation}
In particular, the random variables $\sup_{k:\, t_k^{(1,\,\beta)}\leq T}\,\big(-u t_k^{(\lambda,\,\beta)}+j_k^{(\lambda,\,\beta)}\big)$ and $\sup_k \big(-u t_k^{(\lambda,\,\beta)}+j_k^{(\lambda,\,\beta)}\big)$ are a.s.\ finite and positive.

\noindent (b) For $\lambda, \sigma>0$,
\begin{multline}\label{eq:reduc}
\mmp\Big\{\sup_{k:\,t_k^{(\lambda,\,2)}\leq T} \big(\sigma B(t_k^{(\lambda,\,2)})-u t_k^{(\lambda,\,2)}+j_k^{(\lambda,\,2)}\big)\leq x\Big\}\\=\me\exp\Big(-\lambda \int_0^T \frac{{\rm d}t}{(x-\sigma B(t)+ut)^2}\Big)\1_{\{\sup_{s\in [0,\,T]}\,(\sigma B(s)-us)<x\}}
\end{multline}
and
\begin{multline}
\mmp\Big\{\sup_k \big(\sigma B(t_k^{(\lambda,\,2)})-u t_k^{(\lambda,\,2)}+j_k^{(\lambda,\,2)}\big)\leq x\Big\}=\me\exp\Big(-\lambda \int_0^\infty \frac{{\rm d}t}{((x-\sigma B(t)+ut)^+)^2}\Big)\\=\me\exp\Big(-\lambda\int_0^\infty \frac{{\rm d}t}{(x-\sigma B(t)+ut)^2}\Big)\1_{\{\sup_{s\geq 0}\,(\sigma B(s)-us)<x\}}.\label{eq:mix}
\end{multline}
In particular, the random variables $\sup_{k:\,t_k^{(\lambda,\,2)}\leq T} \big(\sigma B(t_k^{(\lambda,\,2)})-u t_k^{(\lambda,\,2)}+j_k^{(\lambda,\,2)}\big)$ and $\sup_k \big(\sigma B(t_k^{(\lambda,\,2)})-u t_k^{(\lambda,\,2)}+j_k^{(\lambda,\,2)}\big)$ are a.s.\ finite and positive.
\end{assertion}
\begin{proof}
(a) We shall prove \eqref{3889} and
\begin{equation}\label{388889}
\mmp\Big\{\sup_{k:\, t_k^{(\lambda,\,1)}\leq T}\big(-u t_k^{(\lambda,\,1)}+j_k^{(\lambda,\,1)}\big)\leq x\Big\}=\Big(\frac{x}{uT+x}\Big)^{\lambda/u},\quad x>0.
\end{equation}
Sending $T\to\infty$ yields \eqref{388} and \eqref{38888}.

The probabilities on the left-hand sides of \eqref{3889} and \eqref{388889} are equal to
\begin{eqnarray*}
\mmp\big\{N^{(\lambda,\,\beta)}\big((t,y): t\leq T, -ut+y>x\big)=0\big\}=\exp\big(-\me N^{(\lambda,\,\beta)}\big((t,y): t\leq T,
-ut+y>x\big)\big)
\end{eqnarray*}
for $\beta\in (1,2]$ and $\beta=1$, respectively, because $N^{(\lambda,\,\beta)}\big((t,y): t\leq T, -ut+y>x\big)$ is a Poisson distributed random variable. Since
\begin{eqnarray*}
\me N^{(\lambda,\,\beta)}\big((t,y): t\leq T, -ut+y>x\big)&=&\int_0^T \int_{[0,\,\infty)}\1_{\{y>ut+x\}}\mu_{\lambda,\,\beta}({\rm d}y){\rm
d}t=\lambda \int_0^T(ut+x)^{-\beta}{\rm d}t\\&=&
\begin{cases}
        u^{-1}(\beta-1)^{-1}\lambda(x^{1-\beta}-(x+uT)^{1-\beta}), &   \text{if} \ \beta\in (1,2],   \\
        u^{-1}\lambda(\log(uT+x)-\log x), & \text{if} \ \beta=1,
\end{cases}
\end{eqnarray*}
\eqref{3889} and \eqref{388889} follow. Letting in \eqref{3889} and \eqref{388} $x\to 0+$ justifies the claims about the a.s.\ positivity.

\noindent (b) Conditioning on $B$ and arguing as in the proof of part (a) we arrive at
$$\mmp\Big\{\sup_{k:\,t_k^{(\lambda,\,2)}\leq T} \big(\sigma B(t_k^{(\lambda,\,2)})-u t_k^{(\lambda,\,2)}+j_k^{(\lambda,\,2)}\big)\leq x\Big\}=\me\exp\Big(-\lambda\int_0^T \frac{{\rm d}t}{((x-\sigma B(t)+ut)^+)^2}\Big).$$ In the case $u=0$ this formula can also be found in Proposition 1 of \cite{Wang:2014}, along with an equivalent representation of the right-hand side. Formula \eqref{eq:reduc} is its analogue in the case $u>0$.

By the strong law of large numbers for a Brownian motion, the integrand in \eqref{eq:reduc} behaves as $(ut)^{-2}$ as $t\to\infty$. Hence, it is integrable on $[0,\infty)$. Sending in \eqref{eq:reduc} $T\to\infty$ and invoking the Lebesgue dominated convergence theorem proves \eqref{eq:mix}. In view of \eqref{eq:expo}, the random variable $\sup_{s\geq 0}\,(\sigma B(s)-us)$ is a.s.\ positive. Hence, letting in the second part of \eqref{eq:mix} $x\to 0+$ and appealing once again to the Lebesgue dominated convergence theorem we conclude that the variable $\sup_k \big(\sigma B(t_k^{(\lambda,\,2)})-u t_k^{(\lambda,\,2)}+j_k^{(\lambda,\,2)}\big)$ is a.s.\ positive. The a.s.\ positivity of $\sup_{k:\,t_k^{(\lambda,\,2)}\leq T} \big(\sigma B(t_k^{(\lambda,\,2)})-u t_k^{(\lambda,\,2)}+j_k^{(\lambda,\,2)}\big)$  follows analogously.
\end{proof}

Also, we state and prove a lemma, which justifies the usage of space $C$ in our distributional limit theorems.
\begin{lemma}\label{lem:cont}
The processes $\big(\sup_{k\geq 0}\,(S_k-auk+\log \eta_{k+1})\big)_{u>0}$, $\Big(\log \sum_{k\geq 0}\eee^{S_k-auk}\eta_{k+1}\Big)_{u>0}$ and the limit processes in Theorems \ref{perturbed_weak0}, \ref{perturbed_weak} and \ref{main_perturbed_weak convergence} are a.s.\ convex, hence a.s.\ continuous.
\end{lemma}
\begin{proof}
Recall that, according to the discussion at the beginning of Section \ref{sec:newtype}, the first two processes (the converging processes in our distributional limit theorems) are a.s.\ finite for each $u>0$. The a.s.\ finiteness of the limit processes, for each $u>0$, follows from \eqref{eq:expo}, \eqref{388} and \eqref{eq:mix}, respectively.

Further, write, for any $\lambda_1, \lambda_2\geq 0$ satisfying $\lambda_1+\lambda_2=1$ and any $u_1, u_2>0$,
\begin{multline*}
\sup_{k\geq 0}\,(S_k-a(\lambda_1u_1+\lambda_2u_2)k+\log \eta_{k+1})=\sup_{k\geq 0}\,(\lambda_1 (S_k-au_1k+\log \eta_{k+1})+\lambda_2(S_k-au_2k+\log \eta_{k+1}))\\\leq \sup_{k\geq 0}\,(\lambda_1 (S_k-au_1k+\log \eta_{k+1}))+ \sup_{k\geq 0}\,(\lambda_2(S_k-au_2k+\log \eta_{k+1}))\\=\lambda_1 \sup_{k\geq 0}\,(S_k-au_1k+\log \eta_{k+1}) +\lambda_2 \sup_{k\geq 0}\,(S_k-au_2 k+\log \eta_{k+1})
\end{multline*}
having utilized subadditivity of the supremum for the inequality. This proves the claim for the first process. The proofs for the limit processes are analogous.

For each $a>0$, the function $u\mapsto \sum_{k\geq 0}\eee^{S_k-auk}\eta_{k+1}$ is the Laplace-Stieltjes transform of an infinite random measure $\mu_a$ defined by $\mu_a(\{ak\}):=\eee^{S_k}\eta_{k+1}$ for $k\in\mn_0$. It is a standard fact, which is secured by H\"{o}lder's inequality, that any Laplace-Stieltjes transform $f$, say is log-convex, that is, $\log f$ is convex.
\end{proof}

We close this section with another auxiliary result.
\begin{lemma}\label{lem:tail}
Let $\beta\in (1,2]$ and $\lambda>0$. With probability one, for each $u>0$,
\begin{equation}\label{eq:1}
\lim_{T\to\infty}\,\sup_{s\geq T}\,(\sigma B(s)-us)=-\infty,
\end{equation}
$$\lim_{T\to\infty} \sup_{k:\, t_k^{(1,\,\beta)}\geq T}\,\big(-u t_k^{(1,\,\beta)}+j_k^{(1,\,\beta)}\big)=-\infty$$ and
\begin{equation}\label{eq:3}
\lim_{T\to\infty} \sup_{k:\,t_k^{(\lambda,\,2)}\geq T} \big(\sigma B(t_k^{(\lambda,\,2)})-u t_k^{(\lambda,\,2)}+j_k^{(\lambda,\,2)}\big)=-\infty.
\end{equation}
\end{lemma}
\begin{proof}
Relation \eqref{eq:1} follows from $$\lim_{T\to\infty}\,\sup_{s\geq T}\,(\sigma B(s)-us)={\lim\sup}_{T\to\infty}\,(\sigma B(T)-uT)=-\infty\quad \text{a.s.},$$ where the last equality is ensured by the strong law of large numbers for a Brownian motion.

Arguing as in the proof of Proposition \ref{prop:marginal} we conclude that $$\mmp\Big\{\sup_{k:\, t_k^{(1,\,\beta)}\geq T}\,\big(-u t_k^{(1,\,\beta)}+j_k^{(1,\,\beta)}\big)\leq x\Big\}=
\begin{cases}
        0, &   \text{for} \ x\leq -uT,   \\
        \exp\big(-(u(\beta-1))^{-1}(uT+x)^{1-\beta}\big), & \text{for} \ x>-uT.
\end{cases}$$
Letting $T\to\infty$ we infer $\lim_{T\to\infty} \sup_{k:\, t_k^{(1,\,\beta)}\geq T}\,\big(-u t_k^{(1,\,\beta)}+j_k^{(1,\,\beta)}\big)=-\infty$ in probability and, by monotonicity, a.s.

Using subadditivity of the supremum yields $$\sup_{k:\,t_k^{(\lambda,\,2)}\geq T} \big(\sigma B(t_k^{(\lambda,\,2)})-u t_k^{(\lambda,\,2)}+j_k^{(\lambda,\,2)}\big)\leq \sup_{s\geq T}\,(\sigma B(s)-us/2)+\sup_{k:\,t_k^{(\lambda,\,2)}\geq T} \big(-u t_k^{(\lambda,\,2)}/2+j_k^{(\lambda,\,2)}\big).$$ According to formula \eqref{388} with $\beta=2$, $\sup_k\, \big(-u t_k^{(\lambda,\,2)}/2+j_k^{(\lambda,\,2)}\big)$ is a.s.\ finite, whence
$$
\lim_{T\to\infty}\,\sup_{k:\,t_k^{(\lambda,\,2)}\geq T} \big(-u t_k^{(\lambda,\,2)}/2+j_k^{(\lambda,\,2)}\big)<\infty\quad \text{a.s.}
$$
This in combination with \eqref{eq:1} proves \eqref{eq:3}.
\end{proof}

\section{Auxiliary results}

Denote by $D$ the Skorokhod space of c\`{a}dl\`{a}g functions defined on $[0,\infty)$. We assume that the space $D$ is endowed with the $J_1$-topology
\begin{lemma}\label{lem:all-time-sup}
For $n\in\mn_0$, let $f_n\in D$ and $\lim_{n\to\infty}f_n=f_0$ on $D$.
Assume that
$$
M_0:=\sup_{t\geq 0}f_0(t)<\infty
$$
and
\begin{equation}\label{eq:double_limsup_condition}
\limsup_{t\to\infty}\limsup_{n\to\infty}f_n(t)<M_0.
\end{equation}
Then
$$
\lim_{n\to\infty}\sup_{t\geq 0}f_n(t)=M_0.
$$
\end{lemma}
\begin{proof}
By~\eqref{eq:double_limsup_condition}, given sufficiently small $\varepsilon>0$, there exist $T(\varepsilon)\geq 0$ and $n_0(\varepsilon)\in\mn$ such that
$$
f_n(t)\leq M_0-\varepsilon,\quad t\geq T(\varepsilon),\quad n\geq n_0(\varepsilon).
$$
By the definition of supremum, there exists $t_0(\varepsilon)\geq 0$ such that
$$
M_0-\varepsilon/2\leq f_0(t_0(\varepsilon))\leq M_0.
$$
In view of the assumption $\lim_{n\to\infty} f_n=f_0$, there exists a sequence $(t_n(\varepsilon))_{n\in\mn}$ such that
$$
\lim_{n\to\infty}t_n(\varepsilon)=t_0(\varepsilon)\quad\text{and}\quad\lim_{n\to\infty} f_n(t_n(\varepsilon))=f_0(t_0(\varepsilon)).
$$
Thus, there exists $n_1(\varepsilon)\in\mn$ such that, for $n\geq n_1(\varepsilon)$,
$$
t_n(\varepsilon)\leq t_0(\varepsilon)+\varepsilon\quad\text{and}\quad f_n(t_n(\varepsilon))\geq f_0(t_0(\varepsilon))-\varepsilon/2\geq  M_0-\varepsilon.
$$
Put $a(\varepsilon):=\max(T(\varepsilon),t_0(\varepsilon)+\varepsilon)$. Combining the fragments together we conclude that, for $n\geq \max(n_0(\varepsilon),n_1(\varepsilon))$,
$$
\sup_{t\geq 0} f_n(t)=\sup_{t\in [0,\,a(\varepsilon)]} f_n(t)
$$
and thereupon
$$
\lim_{n\to\infty}\sup_{t\geq 0}f_n(t)=\lim_{n\to\infty}\sup_{t\in [0,\,a(\varepsilon)]}f_n(t)=\sup_{t\in [0,\,a(\varepsilon)]}f_0(t)\in [M_0-\varepsilon,M_0].
$$
Sending $\varepsilon\to 0+$ completes the proof.
\end{proof}
\begin{rem}\label{rem:cont}
If $f_0$ is continuous, then~\eqref{eq:double_limsup_condition} boils down to
\begin{equation}\label{eq:double_limsup_condition2}
\limsup_{t\to\infty}f_0(t)<M_0.
\end{equation}
\end{rem}

\begin{cor}\label{cor1}
Under the assumption of Lemma~\ref{lem:all-time-sup}, for each $T>0$,
$$
\lim_{n\to\infty}\sup_{t\geq T}f_n(t)=\sup_{t\geq T}f_0(t).
$$
\end{cor}
\begin{proof}
Apply Lemma~\ref{lem:all-time-sup} to the sequence $(f_n(T+\cdot))_{n\in\mn_0}$.
\end{proof}

We shall need Theorem 1.3.17 in \cite{Iksanov:2016} which we state as Proposition \ref{Iks} and a slight extension of Lemma 1.3.18 in \cite{Iksanov:2016} which we state as Lemma \ref{Iks1}.
Let $C[0,\infty)$ be the set of continuous functions defined on $[0,\infty)$ equipped with the locally uniform topology. Denote by $M_p$ the set of point measures
$\nu$ on $[0,\infty)\times (0,\infty]$ which satisfy
\begin{equation*}\label{1_0_1}
\nu([0, r]\times [\delta,\infty])<\infty
\end{equation*}
for all $r>0$ and all $\delta>0$. The set $M_p$ is endowed with
the topology of vague convergence. Define the mapping $\mathcal{F}$ from $D\times
M_p$ to $D$ by
\begin{equation*}
\mathcal{F}(f,\nu)(t):=
\begin{cases}
        \sup_{k:\ \theta_k\leq t}\,(f(\theta_k)+y_k),  & \text{if} \ \theta_k\leq t \ \text{for
some} \ k,\\
        f(0), & \text{otherwise},
\end{cases}
\end{equation*}
where $\nu = \sum_k \varepsilon_{(\theta_k,\,y_k)}$.
\begin{assertion}\label{Iks}
For $j\in\mn$, let $f_j\in D$ and $\nu_j\in M_p$. Assume that
$f_0\in C[0,\infty)$ and
\begin{itemize}

\item $\nu_0(\{0\}\times (0,+\infty])=0$,

\item $\nu_0((r_1,r_2)\times(0,\infty])\geq 1$ for all positive $r_1$ and $r_2$
such that $r_1<r_2$,

\item $\nu_0 = \sum_k \varepsilon_{\big(\theta^{(0)}_k,\,y^{(0)}_k\big)}$ does not have clustered jumps,
that is, $\theta^{(0)}_k\neq \theta^{(0)}_j$ for $k\neq j$.
\end{itemize}

\noindent If $\lim_{j\to\infty} f_j= f_0$ in the $J_1$-topology on $D$ and $\lim_{j\to\infty} \nu_j=\nu_0$ on $M_p$, then
\begin{equation*}
\lim_{j\to\infty} \mathcal{F}(f_j,\nu_j)= \mathcal{F}(f_0,\nu_0)
\end{equation*}
in the $J_1$-topology on $D$.
\end{assertion}
\begin{lemma}\label{Iks1}
Let $T\geq 0$, $\gamma,\rho>0$. The random measure $N^{(\gamma,\,\rho)}_T:=\sum_k\1_{\{t_k^{(\gamma,\,\rho)}\geq T\}}\varepsilon_{(t_k^{(\gamma,\,\rho)}-T,\,j_k^{(\gamma,\,\rho)})}$ satisfies with probability one all the assumptions imposed on $\nu_0$ in Proposition \ref{Iks}. Here, $(t_k^{(\gamma,\,\rho)},\,j_k^{(\gamma,\,\rho)})$ are the atoms of a Poisson random measure $N^{(\gamma,\,\rho)}$ defined in the paragraph preceding Theorem \ref{perturbed_weak}.
\end{lemma}
\begin{proof}
The case $T=0$ is covered by Lemma 1.3.18 in \cite{Iksanov:2016}. If $T>0$, then $N^{(\gamma,\,\rho)}_T$ is just a deterministic shift of $N^{(\gamma,\,\rho)}$. Since the latter does not have atoms on any fixed deterministic vertical line with probability one, the claim follows.
\end{proof}

Hereafter, we write ${\overset{{\rm f.d.}}\longrightarrow}$ and $\overset{\mmp}{\longrightarrow}$ to denote weak convergence of finite-dimensional distributions and convergence in probability, respectively.
\begin{assertion}\label{cor:bm1}
Under the assumptions of Theorem \ref{perturbed_weak0}, for any $T>0$,
\begin{equation}\label{eq:finsup}
\big(a\sup_{0\leq k\leq \lfloor Ta^{-2}\rfloor}\,(S_k-auk+\log \eta_{k+1})\big)_{u\in\mr}~\overset{{\rm f.d.}}{\longrightarrow}~\big(\sup_{s\in [0,\,T]}(\sigma B(s)-us)\big)_{u\in\mr},\quad a\to 0+,
\end{equation}
where $(B(s))_{s\geq 0}$ is a standard Brownian motion, and, for any $T\geq 0$,
\begin{equation}\label{eq:finsup2}
(a\sup_{k \geq \lfloor Ta^{-2}\rfloor}\,(S_k-auk+\log \eta_{k+1}))_{u>0}~\overset{{\rm f.d.}}{\longrightarrow}~(\sup_{s\geq T}\,(\sigma B(s)-us))_{u>0},\quad a\to 0+.
\end{equation}
\end{assertion}
\begin{proof}
We shall write $\zeta$ for $\log \eta$ and $\zeta_k$ for $\log \eta_k$, $k\in\mn$.

By Donsker's theorem,
$$(aS_{\lfloor Ta^{-2}\rfloor})_{T\geq 0}~\Longrightarrow~(\sigma B(T))_{T\geq 0},\quad a\to 0+,$$ on $D$. Fix any $T>0$. Since, for all $\varepsilon>0$,
\begin{eqnarray*}
\mmp\{a\max_{1\leq k\leq \lfloor Ta^{-2}\rfloor+1}\,\zeta_k>\varepsilon\}&=&1-\big(\mmp\{\zeta \leq\varepsilon a^{-1}\}\big)^{\lfloor Ta^{-2}\rfloor+1}\\&\leq& (\lfloor Ta^{-2}\rfloor+1)\mmp\{\zeta>\varepsilon a^{-1}
\}\to 0
\end{eqnarray*}
as $a\to 0+$ in view of \eqref{cond1}, we infer
$$a\max_{1\leq k\leq \lfloor Ta^{-2}\rfloor+1}\,\zeta_k~\overset{\mmp}{\longrightarrow}~ 0,\quad a\to 0+,$$ which implies
$$(a \zeta_{\lfloor Ta^{-2}\rfloor+1})_{T\geq 0}~\Longrightarrow~ (\Xi(T))_{T\geq 0},\quad a\to 0+$$ on $D$ where $\Xi(t)=0$ for $t\geq 0$. Hence,
$$(a(S_{\lfloor Ta^{-2}\rfloor}+\zeta_{\lfloor Ta^{-2}\rfloor +1}))_{T\geq 0}~\Longrightarrow~(\sigma B(T))_{T\geq 0},\quad a\to 0+
$$ by Slutsky's lemma and thereupon, for any $n\in\mn$ and any $-\infty<u_1<\ldots<u_n<\infty$,
\begin{multline}
(a(S_{\lfloor Ta^{-2}\rfloor}- au_1\lfloor Ta^{-2}\rfloor+\zeta_{\lfloor Ta^{-2}\rfloor +1}),\ldots, a(S_{\lfloor Ta^{-2}\rfloor}- au_n\lfloor Ta^{-2}\rfloor+\zeta_{\lfloor Ta^{-2}\rfloor +1}) )_{T\geq 0}\\~\Longrightarrow~ (\sigma B(T)-u_1 T, \ldots, \sigma B(T)-u_nT)_{T\geq 0},\quad a\to 0+ \label{eq:finitedim}
\end{multline}
in the $J_1$-topology on $D^n$. The supremum functional is continuous in the $J_1$-topology. This in combination with the continuous
mapping theorem proves \eqref{eq:finsup}.

By the Skorokhod representation theorem, there exist versions of the processes in \eqref{eq:finitedim}, for which \eqref{eq:finitedim} holds a.s., with $a$ replaced by $(a_k)_{k\in\mn}$ an arbitrary convergent to $0$ sequence of positive numbers. Each coordinate of the version of the limit is a.s.\ continuous and, for each $u>0$, $\lim_{s\to\infty}(\sigma B(s)-us)=-\infty$ a.s.\ by the strong law of large numbers for a Brownian motion. The latter ensures that the a.s.\ counterpart of \eqref{eq:double_limsup_condition2} holds for each coordinate of the version of the limit. Applying Corollary \ref{cor1} and Remark \ref{rem:cont} separately to each coordinate of the versions and passing from the versions to the original processes we arrive at \eqref{eq:finsup2}.
\end{proof}
\begin{assertion}\label{thm:ppp}
Under the assumptions of Theorem \ref{perturbed_weak}, for any $T>0$,
\begin{equation}\label{eq:ppp_claim11}
\Big(\frac{1}{ac(a)}\sup_{k\leq \lfloor Tc(a)\rfloor }(S_k-auk+\log \eta_{k+1})\Big)_{u>0}~\overset{{\rm f.d.}}{\longrightarrow}~\big(\sup_{k:\,t^{(1,\,\beta)}_k\leq T}(-ut^{(1,\,\beta)}_k+j_k^{(1,\,\beta)})\big)_{u>0},\quad a\to 0+
\end{equation}
and, for any $T\geq 0$,
\begin{equation}\label{eq:ppp_claim2}
\Big(\frac{1}{ac(a)}\sup_{k\geq \lfloor Tc(a)\rfloor}\,(S_k-auk+\log \eta_{k+1})\Big)_{u>0}~\overset{{\rm f.d.}}{\longrightarrow}~ \big(\sup_{k:\;t^{(1,\,\beta)}_k \geq T}(-ut^{(1,\,\beta)}_k+j_k^{(1,\,\beta)})\big)_{u>0},\quad a\to 0+.
\end{equation}

\noindent Under the assumptions of Theorem \ref{main_perturbed_weak convergence}, for any $T>0$,
\begin{equation}\label{eq:ppp_bm_claim11}
\big(a\sup_{k\leq \lfloor Ta^{-2}\rfloor}(S_k-auk+\log \eta_{k+1})\big)_{u>0}~\overset{{\rm f.d.}}{\longrightarrow}~\big(\sup_{k,\,t^{(\lambda,\,2)}_k\geq T}(\sigma B(t^{(\lambda,\,2)}_k)-ut^{(\lambda,\,2)}_k+j_k^{(\lambda,\,2)})\big)_{u>0},\quad a\to 0+
\end{equation}
and, for any $T\geq 0$,
\begin{equation}\label{eq:ppp_bm_claim2}
\big(a\sup_{k\geq \lfloor Ta^{-2}\rfloor}(S_k-auk+\log \eta_{k+1})\big)_{u>0}~\overset{{\rm f.d.}}{\longrightarrow}~\big(\sup_{k:\;t^{(\lambda,\,2)}_k \geq T}(\sigma B(t^{(\lambda,\,2)}_k)-ut^{(\lambda,\,2)}_k+j_k^{(\lambda,\,2)})\big)_{u>0},\quad a\to 0+.
\end{equation}
\end{assertion}
\begin{proof}
In the setting of Theorem \ref{perturbed_weak} put $c(a):=a^{-2}$ for $a>0$.

Fix $u>0$ and $T\geq 0$. Under the assumptions of Theorem \ref{main_perturbed_weak convergence}, there is a joint convergence
\begin{multline}\label{eq:joint_conv_bm_T}
\Big(\Big(\frac{S_{\lfloor (t+T) c(a)\rfloor}}{\sqrt{c(a)}}-u(t+T)\Big)_{t\geq 0},\sum_{k\geq 0}\1_{\{k\geq \lfloor Tc(a)\rfloor\}}\varepsilon_{(k/c(a)-T,\, \zeta_{k+1}/\sqrt{c(a)})}\Big)\\
~\Longrightarrow~\Big((\sigma B(t+T)-u(t+T))_{t\geq 0},
\sum_{k}\1_{\{t_k^{(\lambda,\,2)}\geq T\}}\varepsilon_{(t_k^{(\lambda,\,2)}-T,\,j_k^{(\lambda,\,2)})}\Big),\quad a\to 0+
\end{multline}
in the space $D\times M_p$ endowed with the product topology, see the bottom of p.~27 in \cite{Iksanov:2016}. Moreover, the components on the right-hand side are independent. Fix now any $n\in\mn$ and any $0<u_1<u_2<\ldots<u_n<\infty$. Then \eqref{eq:joint_conv_bm_T} immediately extends to
\begin{multline}\label{eq:joint_conv_bm_T1}
\Big(\Big(\Big(\frac{S_{\lfloor (t+T) c(a)\rfloor}}{\sqrt{c(a)}}-u_i(t+T)\Big)_{t\geq 0}\Big)_{1\leq i\leq n}, \sum_{k\geq 0}\1_{\{k\geq \lfloor Tc(a)\rfloor\}}\varepsilon_{(k/c(a)-T,\, \zeta_{k+1}/\sqrt{c(a)})}\Big)\\
~\Longrightarrow~\Big(((\sigma B(t+T)-u_i(t+T))_{t\geq 0})_{1\leq i\leq n},
\sum_{k}\1_{\{t_k^{(\lambda,\,2)}\geq T\}}\varepsilon_{(t_k^{(\lambda,\,2)}-T,\,j_k^{(\lambda,\,2)})}\Big),\quad a\to 0+
\end{multline}
in the space $D^n\times M_p$ endowed with the product topology, because the components indexed by $i$ only differ by a deterministic term. Similarly, under the assumptions of Theorem~\ref{perturbed_weak},
\begin{multline}\label{eq:joint_conv_zero_T}
\Big(\Big(\Big(\frac{S_{\lfloor (t+T) c(a)\rfloor}}{ac(a)}-u_i(t+T)\Big)_{t\geq 0}\Big)_{1\leq i\leq n},\sum_{k\geq 0}\1_{\{k\geq \lfloor Tc(a)\rfloor\}}\varepsilon_{(k/c(a)-T,\,\zeta_{k+1}/ac(a))}\Big)\\
~\Longrightarrow~\Big(((-u_i(t+T))_{t\geq 0})_{1\leq i\leq n},\sum_k \1_{\{t_k^{(1,\,\beta)}\geq T\}}\varepsilon_{(t_k^{(1,\,\beta)}-T,\,j_k^{(1,\,\beta)})}\Big),\quad a\to 0+
\end{multline}
in the space $D^n\times M_p$ endowed with the product topology, where the convergence of the normalized random walk to the zero process $\Xi$ follows from~\eqref{eq:a2c(a)} which ensures that $ac(a)/\sqrt{c(a)}=a\sqrt{c(a)}\to \infty$ as $a\to 0+$.

Fix any $T_1>T$ and let $(a_j)_{j\in\mn}$ be any sequence of positive numbers satisfying $\lim_{j\to\infty}a_j=0$. By the Skorokhod representation theorem there are versions of the processes, for which \eqref{eq:joint_conv_bm_T1} and \eqref{eq:joint_conv_zero_T} hold a.s. Retaining the original notation for these versions we intend to apply Proposition \ref{Iks} $n$ times with
\begin{multline*}
f_j(t):=\frac{S_{\lfloor (t+T) c(a_j)\rfloor}}{\sqrt{c(a_j)}}-u_i (t+T),\quad \nu_j:=\sum_{k\geq 0}\1_{\{k\geq \lfloor Tc(a_j)\rfloor\}}\varepsilon_{(k/c(a_j)-T,\, \zeta_{k+1}/\sqrt{c(a_j)})}\\
f_0(t):=\sigma B(t+T)-u_i (t+T),\quad \nu_0:=\sum_{k}\1_{\{t_k^{(\lambda,\,2)}\geq T\}}\varepsilon_{(t_k^{(\lambda,\,2)}-T,\,j_k^{(\lambda,\,2)})},\;\; j\in\mn,\;t\geq 0,\; i=1,\ldots,n
\end{multline*}
and $n$ times with
\begin{multline*}
f_j(t):=\frac{S_{\lfloor (t+T) c(a_j)\rfloor}}{a_jc(a_j)}-u_i (t+T),\quad \nu_j:=\sum_{k\geq 0}\1_{\{k\geq \lfloor Tc(a_j)\rfloor\}}\varepsilon_{(k/c(a_j)-T,\,\zeta_{k+1}/a_jc(a_j))}\\
f_0(t):=-u_i (t+T),\quad \nu_0:=\sum_{k}\1_{\{t_k^{(1,\,\beta)}\geq T\}}\varepsilon_{(t_k^{(1,\,\beta)}-T,\,j_k^{(1,\,\beta)})},\;\;j\in\mn,\;t\geq 0,\;i=1,\ldots,n.
\end{multline*}
The so defined converging and limit processes satisfy the assumptions of Proposition \ref{Iks} with probability one. In particular, a.s.\ continuity of the limit functions $f_0$ is obvious, whereas Lemma \ref{Iks1} justifies the claim for the random measures $\nu_0$. A specialization of Proposition \ref{Iks} to the one-dimensional (rather than functional) convergence in conjunction with \eqref{eq:joint_conv_bm_T1} and \eqref{eq:joint_conv_zero_T} yields 
\begin{multline}\label{eq:ppp_bm_claim113}
\Big(a\sup_{\lfloor Tc(a)\rfloor \leq k\leq \lfloor T_1 c(a)\rfloor}\big(S_k-auk+\zeta_{k+1}\big)\Big)_{u>0}\\~\overset{{\rm f.d.}}{\longrightarrow}~\big(\sup_{k:\; T\leq t^{(\lambda,\,2)}_k\leq T_1}(\sigma B(t^{(\lambda,\,2)}_k)-ut^{(\lambda,\,2)}_k+ j_k^{(\lambda,\,2)})\Big)_{u>0},\quad a\to 0+
\end{multline}
and
\begin{multline}\label{eq:ppp_claim114}
\Big(\frac{1}{ac(a)}\sup_{\lfloor Tc(a)\rfloor \leq k\leq \lfloor T_1c(a)\rfloor}\big(S_k-auk+\zeta_{k+1}\big)\Big)_{u>0}\\~\overset{{\rm f.d.}}{\longrightarrow}~\big(\sup_{k:\; T\leq t^{(1,\,\beta)}_k\leq T_1}(-ut^{(1,\,\beta)}_k+j_k^{(1,\,\beta)})\big)_{u>0},\quad a\to 0+,
\end{multline}
respectively. Putting in \eqref{eq:ppp_bm_claim113} and \eqref{eq:ppp_claim114} $T=0$ and then replacing $T_1$ with $T$ we obtain \eqref{eq:ppp_bm_claim11} and \eqref{eq:ppp_claim11}.

The right-hand sides of~\eqref{eq:ppp_bm_claim113} and~\eqref{eq:ppp_claim114} converge a.s.\ as $T_1\to\infty$ to the right-hand sides of~\eqref{eq:ppp_bm_claim2} and~\eqref{eq:ppp_claim2}, respectively. According to Theorem 4.2 on p.~25 in \cite{Billingsley:1968}, both \eqref{eq:ppp_claim2} and \eqref{eq:ppp_bm_claim2} follow if we can show that, for all $\varepsilon>0$, 
$$\lim_{T_1\to\infty}\limsup_{a\to 0+}\mathbb{P}\Big\{\sum_{i=1}^n \big(\sup_{\lfloor Tc(a)\rfloor \leq k}\big(S_k-auk+\zeta_{k+1}\big)-\sup_{\lfloor Tc(a)\rfloor \leq k\leq \lfloor T_1c(a)\rfloor}\big(S_k-auk+\zeta_{k+1}\big)\big)^2>\varepsilon\Big\}=0.$$ Plainly, it is sufficient to check that, for each fixed $u>0$,
$$
\lim_{T_1\to\infty}\limsup_{a\to 0+}\mathbb{P}\big\{\sup_{\lfloor Tc(a)\rfloor \leq k\leq \lfloor T_1c(a)\rfloor}\big(S_k-auk+\zeta_{k+1}\big) \neq \sup_{\lfloor Tc(a)\rfloor \leq k}\big(S_k-auk+\zeta_{k+1}\big)\big\}=0.
$$
Note that, for $T_1\geq 2T$ and $z\in\mr$,
\begin{align*}
&\mathbb{P}\big\{\sup_{\lfloor Tc(a)\rfloor \leq k\leq \lfloor T_1c(a)\rfloor}\big(S_k-auk+\zeta_{k+1}\big) \neq \sup_{\lfloor Tc(a)\rfloor \leq k}\big(S_k-auk+\zeta_{k+1}\big)\big\}\\
&=\mathbb{P}\Big\{\frac{1}{ac(a)}\sup_{\lfloor Tc(a)\rfloor \leq k\leq \lfloor T_1c(a)\rfloor}\big(S_k-auk+\zeta_{k+1}\big) < \frac{1}{ac(a)}\sup_{k > \lfloor T_1c(a)\rfloor}\big(S_k-auk+\zeta_{k+1}\big)\Big\}\\
&\leq \mathbb{P}\Big\{\frac{1}{ac(a)}\sup_{\lfloor Tc(a)\rfloor \leq k \leq \lfloor 2Tc(a)\rfloor}\big(S_k-auk+\zeta_{k+1}\big) < \frac{1}{ac(a)}\sup_{k> \lfloor T_1c(a)\rfloor}\big(S_k-auk+\zeta_{k+1}\big)\Big\}\\
&\leq \mathbb{P}\Big\{\frac{1}{ac(a)}\sup_{\lfloor Tc(a)\rfloor \leq k \leq \lfloor 2Tc(a)\rfloor}\big(S_k-auk+\zeta_{k+1}\big)\leq z\Big\}\\&+\mmp\Big\{\frac{1}{ac(a)}\sup_{k > \lfloor T_1c(a)\rfloor}\big(S_k-auk+\zeta_{k+1}\big)>z\Big\}.
\end{align*}
According to \eqref{eq:ppp_bm_claim113} and \eqref{eq:ppp_claim114}, the random variables $(ac(a))^{-1}\sup_{\lfloor Tc(a)\rfloor \leq k\leq \lfloor 2Tc(a)\rfloor}\big(S_k-auk+\zeta_{k+1}\big)$ converge in distribution as $a\to 0+$ to an a.s.~finite random variable $\rho_T$, say (the a.s.\ finiteness follows from Proposition \ref{prop:marginal}). As a consequence, the first probability on the right-hand side tends to $0$ as $a\to 0+$ and then $z\to -\infty$ along the sequence of continuity points of the distribution function of $\rho_T$. Thus, it is enough to prove that, for each fixed $z\in\mathbb{R}$,
$$
\lim_{T_1\to\infty}\limsup_{a\to 0+}\mathbb{P}\Big\{\frac{1}{ac(a)}\sup_{k > \lfloor T_1 c(a)\rfloor}\big(S_k-auk+\zeta_{k+1}\big)>z\Big\}=0.$$
The latter probability does not exceed
\begin{equation}\label{eq:proof1}
\mathbb{P}\big\{\sup_{k > \lfloor T_1 c(a)\rfloor}(S_k-auk/2)> ac(a)\big\}+\mathbb{P}\big\{\sup_{k > \lfloor T_1 c(a)\rfloor}(\zeta_{k+1}-auk/2)>(z-1) ac(a)\big\}.
\end{equation}
In the setting of Theorem~\ref{main_perturbed_weak convergence},
$$
\lim_{T_1\to\infty}\limsup_{a\to 0+}\mathbb{P}\big\{\sup_{k > \lfloor T_1 c(a)\rfloor}(S_k-auk/2)> ac(a)\big\}=0.
$$
by \eqref{eq:finsup2} and \eqref{eq:1}, whereas in the setting of Theorem~\ref{perturbed_weak} this follows from \eqref{eq:finsup2} and \eqref{eq:a2c(a)}. Indeed, for small enough $a>0$, \eqref{eq:a2c(a)} entails $c(a)\geq a^{-2}$, whence
$$
\mathbb{P}\big\{\sup_{k > \lfloor T_1 c(a)\rfloor}(S_k-auk/2)> ac(a)\big\}\leq \mathbb{P}\big\{a\sup_{k > \lfloor T_1 a^{-2}\rfloor}(S_k-auk/2)> a^2c(a)\big\}.
$$
In view of~\eqref{eq:finsup2}, the right-hand side converges to $0$ as $a\to 0+$.

As far as the second summand in~\eqref{eq:proof1} is concerned we argue as follows. For large $T_1>\max(2(1-z),0)$ and small $a>0$,
\begin{multline*}
\mathbb{P}\big\{\sup_{k > \lfloor T_1 c(a)\rfloor}(\zeta_{k+1}-auk/2)> (z-1)ac(a)\big\}\leq \sum_{k>\lfloor T_1 c(a)\rfloor }\mathbb{P}\left\{\zeta_{k+1}>auk/2+(z-1)ac(a)-a/2\right\}\\
\leq \int_{\lfloor T_1c(a)\rfloor}^\infty \mathbb{P}\{\zeta> aux/2+(z-1)ac(a)-a/2\}{\rm d}x
=\frac{2}{ua}\int_{(a/2)(\lfloor T_1c(a)\rfloor u+2(z-1)c(a)-1)}^\infty \mathbb{P}\{\zeta>x\}{\rm d}x\\
\sim~\frac{1}{ua}\frac{1}{\beta-1}(T_1 u+ 2(z-1)) ac(a)\mathbb{P}\{\zeta> (T_1u+2(z-1))ac(a)/2\}\\
\sim~\frac{1}{au}\frac{1}{\beta-1}(T_1u+ 2(z-1))ac(a)\mathbb{P}\{\zeta> (T_1u+2(z-1))b(c(a))/2\}\\
\sim~\frac{1}{au}\frac{2^\beta}{\beta-1}(T_1u+2(z-1))^{1-\beta} ac(a)\frac{1}{c(a)},\quad a\to 0+.
\end{multline*}
We have used Proposition 1.5.10 in \cite{Bingham+Goldie+Teugels:1989} for the first asymptotic equivalence. Since $\beta>1$, the right-hand side converges to zero as $T_1\to\infty$. This completes the proof of Proposition \ref{thm:ppp}.
\end{proof}

\section{Proofs of Theorems \ref{perturbed_weak0}, \ref{perturbed_weak} and \ref{main_perturbed_weak convergence}}
We shall prove all the results simultaneously. To this end, for $a,T, u>0$,  put $$m(a):=a,~c(a):=a^{-2},~ X_1(u,T):=\sup_{s\in [0,\,T]}\,(\sigma B(s)-us), ~ X_1(u, \infty):=\sup_{s\geq 0}\,(\sigma B(s)-us)$$ and $$X_1^\ast(u,T):=\sup_{s\geq T}\,(\sigma B(s)-us)$$ under the assumptions of Theorem \ref{perturbed_weak0}, $$m(a):=(ac(a))^{-1},~c(a)~\text{is as defined in Theorem \ref{perturbed_weak}},~ X_2(u,T):=\sup_{k:~t_k^{(1,\,\beta)}\leq T}\,\big(-ut_k^{(1,\,\beta)}+j_k^{(1,\,\beta)}\big)$$ $$X_2(u, \infty):=\sup_k\,\big(-ut_k^{(1,\,\beta)}+j_k^{(1,\,\beta)}\big)\quad\text{and}\quad X_2^\ast(u,T):=\sup_{k:~t_k^{(1,\,\beta)}\geq T}\,\big(-ut_k^{(1,\,\beta)}+j_k^{(1,\,\beta)}\big)$$ under the assumptions of Theorem \ref{perturbed_weak} and
$$
m(a):=a,~c(a):=a^{-2},~ X_3(u,T):=\sup_{k:~t_k^{(\lambda,\,2)}\leq T}\,\big(\sigma B(t_k^{(\lambda,\,2)})-u t_k^{(\lambda,\,2)}+j_k^{(\lambda,\,2)}\big)$$
\begin{multline*}
X_3(u, \infty):=\sup_k \big(\sigma  B(t_k^{(\lambda,\,2)})-u t_k^{(\lambda,\,2)}+j_k^{(\lambda,\,2)}\big)\\
\text{and}\quad X_3^\ast(u,T):=\sup_{k:~t_k^{(\lambda,\,2)}\geq T}\,\big(\sigma B(t_k^{(\lambda,\,2)})-u t_k^{(\lambda,\,2)}+j_k^{(\lambda,\,2)}\big)
\end{multline*}
under the assumptions of Theorem \ref{main_perturbed_weak convergence}.

Since the converging processes
$$
\big(m(a)\sup_{k\geq 0}\, (S_k-auk+\log \eta_{k+1})\big)_{u>0}\quad\text{and}\quad (m(a)\log \sum_{k\geq 0}\eee^{S_k-auk}\eta_{k+1})_{u>0}
$$
are a.s.\ nonincreasing and, by Lemma \ref{lem:cont}, the limit processes $(X_l(u, \infty))_{u>0}$, $l=1,2,3$ are a.s.\ continuous, weak convergence of probability measures in $C$ is equivalent to weak convergence of the corresponding finite-dimensional distributions. This follows from Skorokhod's representation theorem in combination with Dini's theorem.

Thus, limit relations \eqref{eq:supr}, \eqref{limit_perturbed1111} and \eqref{limit_suprem}, dealing with the convergence of suprema, are ensured by \eqref{eq:finsup2}, \eqref{eq:ppp_claim2} and \eqref{eq:ppp_bm_claim2} all with $T=0$, respectively, and the last remark. As far as the perpetuities are concerned, we have to show that
\begin{equation}\label{nocut}
\Big(m(a)\log \sum_{k\geq 0} \eee^{S_k-auk}\eta_{k+1}\Big)_{u>0}~\overset{{\rm f.d.}}{\longrightarrow}~(X_l(u,\infty))_{u>0},\quad a\to 0+,\quad l=1,2,3.
\end{equation}
As a preparation, we prove that, for any $T>0$,
\begin{equation}\label{cut}
\Big(m(a)\log \sum_{k=0}^{\lfloor Tc(a)\rfloor}\eee^{S_k-auk}\eta_{k+1}\Big)_{u>0}~\overset{{\rm f.d.}}{\longrightarrow}~(X_l(u,T))_{u>0},\quad a\to 0+,\quad l=1,2,3.
\end{equation}

Fix any $n\in\mn$, any $\gamma_1,\ldots,\gamma_n \in \mathbb{R}$ and any $0<u_1,u_2,\ldots,u_n<\infty$. Assume, without loss of generality, that $\gamma_1,\ldots, \gamma_{n_0}\geq 0$ and $\gamma_{n_0+1},\ldots, \gamma_n<0$ for some $n_0\in\mn_0$, $n_0\leq n$. In particular, the situation is allowed in which all $\gamma_j$ are of the same sign (in which case $n_0=0$ or $n_0=n$). In view of the Cram\'{e}r-Wold device, relation \eqref{cut} is equivalent to the following: for any $T>0$, as $a\to 0+$,
\begin{equation}\label{cut1}
m(a)\sum_{j=1}^n\gamma_j\log \sum_{k=0}^{\lfloor Tc(a)\rfloor}\eee^{S_k-au_jk}\eta_{k+1}{\overset{{\rm d}}\longrightarrow}\sum_{j=1}^n\gamma_j X_l(u_j,T),\quad l=1,2,3.
\end{equation}
To prove \eqref{cut1}, write, for any $T>0$,
\begin{multline*}
\sum_{j=1}^n\gamma_j\log \sum_{k=0}^{\lfloor Tc(a)\rfloor}\eee^{S_k-au_jk}\eta_{k+1}=\sum_{j=1}^{n_0} \gamma_j\log \sum_{k=0}^{\lfloor Tc(a)\rfloor}\eee^{S_k-au_jk}\eta_{k+1}+\sum_{j=n_0+1}^n \gamma_j\log \sum_{k=0}^{\lfloor Tc(a)\rfloor}\eee^{S_k-au_jk}\eta_{k+1}\\\leq \sum_{j=1}^{n_0} \gamma_j \big( \log (\lfloor Tc(a)\rfloor+1)+\max_{0\leq k\leq \lfloor Tc(a)\rfloor }\,(S_k-au_jk+\zeta_{k+1})\big)\\+\sum_{j=n_0+1}^n \gamma_j\max_{0\leq k\leq \lfloor Tc(a)\rfloor }\,(S_k-au_jk+\zeta_{k+1}),
\end{multline*}
where $\zeta_j=\log \eta_j$ for $j\in\mn$, and analogously
\begin{multline*}
\sum_{j=1}^n\gamma_j\log \sum_{k=0}^{\lfloor Tc(a)\rfloor}\eee^{S_k-au_jk}\eta_{k+1}\geq \sum_{j=1}^{n_0} \gamma_j \max_{0\leq k\leq \lfloor Tc(a)\rfloor }\,(S_k-au_jk+\zeta_{k+1})\\+\sum_{j=n_0+1}^n \gamma_j\big( \log (\lfloor Tc(a)\rfloor+1)+\max_{0\leq k\leq \lfloor Tc(a)\rfloor }\,(S_k-au_jk+\zeta_{k+1})\big).
\end{multline*}
With these at hand, \eqref{cut1} follows from
\begin{equation}\label{supr}
(m(a)\sup_{0\leq k\leq \lfloor Tc(a)\rfloor}\,(S_k-auk+\zeta_{k+1}))_{u>0}~\overset{{\rm f.d.}}{\longrightarrow}~(X_l(u,T))_{u>0},\quad a\to 0+,\quad l=1,2,3,
\end{equation}
(see \eqref{eq:finsup}, \eqref{eq:ppp_claim11} and \eqref{eq:ppp_bm_claim11}) and the fact that $\lim_{a\to 0+}m(a)\log c(a)=0$. In the setting of Theorem \ref{perturbed_weak0} the latter is justified by the regular variation of $c$ at $0+$ of index $-\beta/(\beta-1)$, see Remark \ref{rem1}. This particularly implies that the function $a\mapsto m(a)=(ac(a))^{-1}$ is regularly varying at $0+$ of positive index $(\beta-1)^{-1}$.

Plainly, $\lim_{T\to\infty} \sum_{j=1}^n\gamma_j X_l(u_j,T)=\sum_{j=1}^n\gamma_j X_l(u_j,\infty)$ a.s. Hence, according to Theorem 4.2 on p.~25 in \cite{Billingsley:1968} the proof of \eqref{nocut} is complete if we can show that, for all $\varepsilon>0$, $$\lim_{T\to\infty}{\lim\sup}_{a\to 0+}\mmp\Big\{m(a) \Big|\sum_{j=1}^n\gamma_j\Big(\log \sum_{k\geq 0} \eee^{S_k-au_jk}\eta_{k+1}-\log \sum_{k=0}^{\lfloor Tc(a)\rfloor}\eee^{S_k-au_jk}\eta_{k+1}\Big)\Big|>\varepsilon \Big\}=0.$$ By the triangle inequality, it is enough to prove that, with $u>0$ fixed, $$\lim_{T\to\infty}{\lim\sup}_{a\to 0+}\mmp\Big\{m(a)\Big(\log \sum_{k\geq 0} \eee^{S_k-auk}\eta_{k+1}-\log \sum_{k=0}^{\lfloor Tc(a)\rfloor}\eee^{S_k-auk}\eta_{k+1}\Big)>\varepsilon \Big\}=0.$$ The latter probability is upper bounded as follows:
\begin{multline*}
\leq \mmp\Big\{m(a)\Big(\log^+\sum_{k\geq 0} \eee^{S_k-auk}\eta_{k+1}-\log \sum_{k=0}^{\lfloor Tc(a)\rfloor}\eee^{S_k-auk}\eta_{k+1}\Big)>\varepsilon,\,\log \sum_{k=0}^{\lfloor T c(a)\rfloor}\eee^{S_k-auk}\eta_{k+1}\leq 0\Big\}\\+\mmp\Big\{m(a)\Big(\log^+\sum_{k\geq 0} \eee^{S_k-auk}\eta_{k+1}-\log \sum_{k=0}^{\lfloor Tc(a)\rfloor}\eee^{S_k-auk}\eta_{k+1}\Big)>\varepsilon,\,\log \sum_{k=0}^{\lfloor Tc(a)\rfloor}\eee^{S_k-auk}\eta_{k+1}>0\Big\}\\\leq \mmp\Big\{m(a)\log \sum_{k=0}^{\lfloor Tc(a)\rfloor}\eee^{S_k-auk}\eta_{k+1}\leq 0\Big\}\\+\mmp\Big\{m(a)\Big(\log^+\sum_{k\geq 0} \eee^{S_k-auk}\eta_{k+1}-\log^+ \sum_{k=0}^{\lfloor Tc(a)\rfloor}\eee^{S_k-auk}\eta_{k+1}\Big)>\varepsilon\Big\}.
\end{multline*}
The first probability on the right-hand side converges to $0$ as $a\to 0+$. This is secured by \eqref{cut1} with $n=1$ and $\gamma_1=1$, and the fact that the right-hand sides in \eqref{cut1} are a.s.\ positive. The latter follows from \eqref{eq:expo} and Proposition \ref{prop:marginal}. To proceed, we need two inequalities:
\begin{equation}\label{ineq2}
\log^+(x+y)\leq \log^+(x)+\log^+(y)+2\log 2,\quad x,y\geq 0
\end{equation}
and
\begin{equation}\label{ineq5000}
\log^+(xy)\leq \log^+x +\log^+y,\quad x,y\geq 0.
\end{equation}
Inequality \eqref{ineq2} follows from
$$\log^+ (x)\leq \log (1+x)\leq \log^+ (x)+\log 2,\quad x\geq 0$$ and
the subadditivity of $x\mapsto \log (1+x)$, namely,
\begin{eqnarray*}
\log^+(x+y)&\leq& \log (1+x+y)\leq \log (1+x)+\log (1+y)\notag\\&\leq& \log^+ (x)+\log^+(y)+2\log 2,\quad
x,y\geq 0.
\end{eqnarray*}
Inequality \eqref{ineq5000} is a consequence of
the subadditivity of $x\to x^+$.

In view of \eqref{ineq2}, it remains to prove that
\begin{equation}\label{eq:intermed}
\lim_{T\to\infty}{\lim\sup}_{a\to 0+}\mmp\Big\{m(a)\log^+ \sum_{k>\lfloor Tc(a)\rfloor} \eee^{S_k-auk}\eta_{k+1}>\varepsilon \Big\}=0.
\end{equation}
To this end, write, with the help of \eqref{ineq5000},  $$\log^+ \sum_{k>\lfloor Tc(a)\rfloor} \eee^{S_k-auk}\eta_{k+1}\leq (\sup_{k>\lfloor Tc(a)\rfloor}\,(S_k-auk/2+\zeta_{k+1}))^++\log^+\sum_{k>\lfloor Tc(a)\rfloor} \eee^{-auk/2}.$$ While $\lim_{a\to 0+} \log^+\sum_{k>\lfloor Tc(a)\rfloor} \eee^{-auk/2}=0$, formulae \eqref{eq:finsup2}, \eqref{eq:ppp_claim2} and \eqref{eq:ppp_bm_claim2} entail $$m(a)(\sup_{k>\lfloor Tc(a)\rfloor}\,(S_k-auk/2+\zeta_{k+1}))^+~\overset{{\rm
d}}{\longrightarrow}~(X_l^\ast(u/2,T))^+,\quad a\to 0+,~l=1,2,3.$$ Finally, by Lemma \ref{lem:tail}, $\lim_{T\to\infty}\,(X_l^\ast(u/2,T))^+=0$ a.s., and \eqref{eq:intermed} follows.

The proof of Theorems \ref{perturbed_weak0}, \ref{perturbed_weak} and \ref{main_perturbed_weak convergence} is complete.

\section{Proof of Theorem \ref{thm:lil} and Corollary \ref{corr:lil}}

The following deterministic result is a consequence of Corollary 4.12.5 in \cite{Bingham+Goldie+Teugels:1989}.
\begin{assertion}\label{asser:det}
Let $A\in (0,\infty)$ and $\mu$ be a locally finite measure on $[0,\infty)$. Assume that the function $\varphi$ is regularly varying at $\infty$ of index $\alpha>1$ and put $\psi(t):=\varphi(t)/t$ for large $t$. Then
\begin{equation}\label{asser:claim1}
{\lim\sup}_{x\to\infty}\frac{\log \mu ([0,\,\varphi(x)])}{x}=A
\end{equation}
if, and only if,
\begin{equation}\label{asser:claim2}
{\lim\sup}_{\lambda \to\infty}\frac{\log \int_{[0,\,\infty)}\eee^{-x/\psi(\lambda)} \mu ({\rm d}x)}{\lambda}=(\alpha-1)\Big(\frac{A}{\alpha}\Big)^{\alpha/(\alpha-1)}.
\end{equation}
\end{assertion}
\begin{proof}
Assume that \eqref{asser:claim1} holds. Then, according to the first implication in Corollary 4.12.5 in \cite{Bingham+Goldie+Teugels:1989},
\begin{equation}\label{eq:BGT_cor}
{\lim\sup}_{\lambda \to\infty}\frac{\log \int_{[0,\,\infty)}\eee^{-x/\psi(\lambda)} \mu ({\rm d}x)}{\lambda}\leq (\alpha-1)\Big(\frac{A}{\alpha}\Big)^{\alpha/(\alpha-1)}.
\end{equation}
Suppose that the above inequality is strict, that is, for some $\varepsilon>0$ and $A(\varepsilon)<A$,
$$
{\lim\sup}_{\lambda \to\infty}\frac{\log \int_{[0,\,\infty)}\eee^{-x/\psi(\lambda)} \mu ({\rm d}x)}{\lambda}=(\alpha-1)\Big(\frac{A}{\alpha}\Big)^{\alpha/(\alpha-1)}-\varepsilon=(\alpha-1)\Big(\frac{A(\varepsilon)}{\alpha}\Big)^{\alpha/(\alpha-1)}.
$$
Then the second implication in the aforementioned Corollary 4.12.5 yields
$$
{\lim\sup}_{x\to\infty}\frac{\log \mu ([0,\,\varphi(x)])}{x}\leq A_1(\varepsilon)<A,
$$
which is a contradiction. Thus, the inequality in \eqref{eq:BGT_cor} can be replaced by the equality. The inverse implication follows analogously.
\end{proof}

The following result is needed for the proof of Theorem \ref{thm:lil}(a) and also of independent interest.
\begin{thm}\label{thm:suprLIL}
Suppose that \eqref{10} and \eqref{eq:oneta} hold. Then
\begin{equation}\label{eq:inter1}
{\lim\sup}_{t\to\infty}\frac{\max_{0\leq k\leq \lfloor t\rfloor}\,(S_k+\log \eta_{k+1})}{(t\log\log t)^{1/2}}=2^{1/2}\sigma\quad\text{{\rm a.s.}}
\end{equation}
and
\begin{equation}\label{eq:inter2}
{\lim\sup}_{t\to\infty}\frac{\log \sum_{k=0}^{\lfloor t\rfloor}\eee^{S_k}\eta_{k+1}}{(t\log\log t)^{1/2}}=2^{1/2}\sigma\quad\text{{\rm a.s.}}
\end{equation}
\end{thm}
\begin{proof}
In view of $$\max_{0\leq k\leq \lfloor t\rfloor}\,(S_k+\log \eta_{k+1})\leq \log \sum_{k=0}^{\lfloor t\rfloor}\eee^{S_k}\eta_{k+1}\leq \log(\lfloor t\rfloor+1)+\max_{0\leq k\leq \lfloor t\rfloor}\,(S_k+\log \eta_{k+1})\quad\text{a.s.},$$ it suffices to prove \eqref{eq:inter1}. Furthermore, when doing so we can and do replace $\lfloor t\rfloor$ with integer $n$.

Note that
\begin{equation}\label{eq:lilsupr}
{\lim\sup}_{n\to\infty}\frac{\max_{0\leq k\leq  n}\,S_k}{(n\log\log n)^{1/2}}=2^{1/2}\sigma\quad\text{a.s.},
\end{equation}
see, for instance, p.~439 in \cite{Bingham:1986}. Recall the notation $\zeta=\log \eta$ and $\zeta_k=\log \eta_k$ for $k\in\mn$. The assumption $\me f(\zeta^+)<\infty$ in combination with the Borel-Cantelli lemma entails
$$
\lim_{n\to\infty} (n\log\log n)^{-1/2}\zeta^+_n=0\quad\text{a.s.}
$$
and thereupon $\lim_{n\to\infty} (n\log\log n)^{-1/2}\max_{1\leq k\leq n+1}\, \zeta^+_k=0$ a.s. Using the latter, \eqref{eq:lilsupr} and $$\max_{0\leq k\leq n}\,(S_k+\zeta_{k+1})\leq \max_{0\leq k\leq n}\,S_k+\max_{1\leq k\leq n+1}\,\zeta^+_k\quad \text{a.s.}$$ we infer $${\lim\sup}_{n\to\infty}\frac{\max_{0\leq k\leq n}\,(S_k+\zeta_{k+1})}{(n\log\log n)^{1/2}}\leq 2^{1/2}\sigma \quad\text{a.s.}$$

Fix any $\delta\in (0,1)$, put $x_n(\delta):=(1-\delta)2^{1/2}\sigma (n\log\log n)^{1/2}$ for $n\geq 3$ and define the event $$A_n=A_n(\delta):=\{\max_{0\leq k\leq n}\,(S_k+\zeta_{k+1})>x_n(\delta)\},\quad n\geq 3.$$ Our purpose is to show that $$\mmp\{A_n~~\text{i.o.}\}=1.$$ Here, as usual, `${\rm i.o.}$' is a shorthand for `infinitely often' and $\{A_n~~\text{i.o.}\}=\cap_{n\geq 3}\cup_{k\geq n} A_k$. Pick any $\gamma\in\mathbb{R}$ satisfying $\mmp\{\zeta>\gamma\}>0$. For $n\geq 3$, put $\tau_n:=\inf\{k\leq n: S_k>x_n(\delta)-\gamma\}$ on the event $\{\max_{0\leq k\leq n}\,S_k>x_n(\delta)-\gamma\}$ and $\tau_n:=+\infty$ on the complementary event. Now define the events $$B_n=B_n(\delta):=\{\max_{0\leq k\leq n}\,S_k>x_n(\delta)-\gamma\}\quad\text{and}\quad C_n=C_n(\delta):=\{\zeta_{\tau_n+1}>\gamma\},\quad n\geq 3.$$ Observe that, for each $n\geq 3$, $B_n\cap C_n\subseteq A_n$, whence $$\Big\{\sum_{n\geq 3}\1_{B_n\cap C_n}=\infty\Big\}=\{B_n\cap C_n~~\text{i.o.}\}\subseteq \{A_n~~\text{i.o.}\}.$$ For $n\geq 3$, denote by $\mathcal{F}_n$ the $\sigma$-algebra generated by $((\xi_k, \zeta_k))_{1\leq k\leq n}$. In view of $$\zeta_{\tau_n+1}\1_{\{\tau_n\leq n\}}=\zeta_1\1_{\{x_n(\delta)<\gamma\}}+\sum_{k=1}^n \zeta_{k+1}\1_{\{\max_{0\leq j\leq k-1}\,S_j\leq x_n(\delta)-\gamma,\, S_k>x_n(\delta)-\gamma\}},$$ we conclude that $B_n\cap C_n\in \mathcal{F}_{n+1}$. Hence, by the conditional Borel-Cantelli lemma (see, for instance, Theorem 5.3.2 on p.~240 in \cite{Durrett:2010}), $$\mmp\Big\{\sum_{n\geq 3}\1_{B_n\cap C_n}=\infty\Big\}=\mmp\Big\{\sum_{n\geq 3}\mmp\{B_n\cap C_n|\mathcal{F}_n\}=\infty\Big\}.$$ Since $$\sum_{n\geq 3}\mmp\{B_n\cap C_n|\mathcal{F}_n\}=\mmp\{\theta>\gamma\}\sum_{n\geq 3}\1_{B_n}$$ and  \eqref{eq:lilsupr} secures $\mmp\{\sum_{n\geq 3}\1_{B_n}=\infty\}=1$, we infer $\mmp\{B_n\cap C_n~~\text{i.o.}\}=1= \mmp\{A_n~~\text{i.o.}\}$. The proof of Theorem \ref{thm:suprLIL} is complete.
\end{proof}
\begin{rem}\label{rem2}
If \eqref{10} holds and \eqref{eq:oneta} does not hold, then both \eqref{eq:inter1} and \eqref{eq:inter2} fail to hold. As a consequence, so does \eqref{eq:perpLIL} as follows from Proposition \ref{asser:det}. By the Borel-Cantelli lemma, the condition  $\me f(\zeta^+)=\infty$ entails $${\lim\sup}_{n\to\infty}\frac{\max_{1\leq k\leq n+1}\,\zeta_k}{(n\log\log n)^2}={\lim\sup}_{n\to\infty}\frac{\max_{1\leq k\leq n+1}\,\zeta^+_k}{(n\log\log n)^2}=+\infty\quad\text{a.s.}$$ Using this, relation \eqref{eq:lilsupr}, applied to $(-S_k)$ instead of $S_k$, and $$\max_{0\leq k\leq n}\,(S_k+\zeta_{k+1})\geq \max_{1\leq k\leq n+1}\,\zeta_k-\max_{0\leq k\leq n}\,(-S_k)\quad \text{a.s.}$$ we infer $${\lim\sup}_{n\to\infty}\frac{\max_{0\leq k\leq n}\,(S_k+\zeta_{k+1})}{(n\log\log n)^2}=+\infty\quad\text{a.s.}$$
\end{rem}

\begin{proof}[Proof of Theorem \ref{thm:lil}]
Both parts of the theorem will be proved by an application of Proposition \ref{asser:det}. In addition, we find it instructive to give an alternative, more probabilistic proof of the relation ${\lim\sup}_{a\to 0+}\leq 1$ a.s.\ in part (b) which takes an advantage of formula \eqref{Urban}.

\noindent (b) In the setting of Proposition \ref{asser:det}, let $\mu$ be a {\it random} measure defined by $$\mu([0,\,t]):=\int_0^t \eee^{B(s)}{\rm d}s,\quad t\geq 0,$$ and put $\psi(x):=x/\log\log x$ for $x\geq x_0$, where $x_0>\eee$ is chosen to ensure that $\psi$, hence $x\mapsto f(x)=x^2/\log\log x$, are strictly increasing and continuous on $(x_0, \infty)$. We intend to show that
\begin{equation}\label{eq:inter}
{\lim\sup}_{t\to\infty}\frac{\log \int_0^t \eee^{B(s)}{\rm d}s}{(t\log\log t)^{1/2}}=2^{1/2}\quad\text{a.s.}
\end{equation}
On the one hand, $\log \int_0^t \eee^{B(s)}{\rm d}s\leq \log t+\max_{s\in [0,\,t]}\,B(s)$ a.s. On the other hand, let $\tau_t\in [0,\,t]$ denote any (random) point satisfying $B(\tau_t)=\max_{s\in [0,\,t]}\,B(s)$ a.s. Then, given $\varepsilon>0$ there exists a random $\delta\in (0,1)$ such that $B(u)\geq \max_{s\in [0,\,t]}\,B(s)-\varepsilon$ whenever $u\in (\tau_t-\delta, \tau_t+\delta)\cap (0,\infty)$. This yields $\log \int_0^t \eee^{B(u)}{\rm d}u\geq \log \int_{\max(\tau_t-\delta,\,0)}^{\tau_t+\delta} \eee^{B(u)}{\rm d}u\geq \max_{s\in [0,\,t]}\,B(s)-\varepsilon+\log \delta$ a.s. Now relation \eqref{eq:inter} follows from the two inequalities and $${\lim\sup}_{t\to\infty}\frac{\max_{s\in [0,\,t]}\,B(s)}{(t\log\log t)^{1/2}}=2^{1/2}\quad\text{a.s.}$$ For the latter, see, for instance, p.~439 in \cite{Bingham:1986}.

Formula \eqref{eq:inter} entails an a.s.\ version of \eqref{asser:claim1}, with the present choice of $\mu$, $\varphi=f$, $A=2^{1/2}$ and $\alpha=2$. By Proposition \ref{asser:det}, \eqref{asser:claim2} holds with the right-hand side being equal to $1/2$. Hence, $${\lim\sup}_{a\to 0+} \frac{\log \int_{[0,\,\infty)}\eee^{-ax} \mu ({\rm d}x)}{\psi^\leftarrow(1/a)}=2^{-1}\quad \text{a.s.},$$ where $\psi^\leftarrow$ the generalized inverse function of $\psi$ satisfies $\psi^\leftarrow(t)\sim t\log\log t$ as $t\to\infty$. This proves \eqref{eq:brm}.

Here is the promised probabilistic proof. Fix any $r>1$. In view of \eqref{Urban}, a distribution density of $r^{-n}\big(\log\int_0^\infty \eee^{B(s)-r^{-(n+1)}s}{\rm d}s-\log 2\big)$ is $$x\mapsto \frac{r^n\eee^{-2r^{-1}x}\exp(-\eee^{-r^n x})}{\Gamma(2r^{-(n+1)})},\quad x\to\infty.$$ Hence, for all $\varepsilon>0$,
\begin{multline*}
\mmp\Big\{r^{-n}\Big(\log\int_0^\infty \eee^{B(s)-r^{-(n+1)}s}{\rm d}s-\log 2\Big)> 2^{-1}r(1+\varepsilon)\log \log r^n\Big\}\\=\frac{r^n}{\Gamma(2r^{-(n+1)})}\int_{2^{-1}r(1+\varepsilon)\log (n\log r)}^\infty \eee^{-2r^{-1}x}\exp(-\eee^{-r^nx}){\rm d}x\\\leq \frac{r^n}{\Gamma(2r^{-(n+1)})}\int_{2^{-1}r(1+\varepsilon)\log (n\log r)}^\infty \eee^{-2r^{-1}x}{\rm d}x=\frac{1}{2 r^{-(n+1)}\Gamma(2r^{-(n+1)})\eee^{(1+\varepsilon)\log (n\log r)}}~\sim~ \frac{1}{n^{1+\varepsilon}}
\end{multline*}
as $n\to\infty$ having utilized $\lim_{x\to 0+}x\Gamma(x)=1$. Thus, by the Borel-Cantelli lemma, $${\lim\sup}_{n\to\infty}\frac{2\log\int_0^\infty \eee^{B(s)-r^{-(n+1)}s}{\rm d}s}{r^n \log \log r^n}\leq r \quad \text{a.s.}$$ For each $a\in (0,1]$ there exists $n\in\mn_0$ such that $a\in [r^{-(n+1)}, r^{-n}]$ and, by monotonicity, for such $a$, $$\frac{2a \log\int_0^\infty \eee^{B(s)-as}{\rm d}s}{\log \log (1/a)}\leq \frac{2 \log\int_0^\infty \eee^{B(s)-r^{-(n+1)}s}{\rm d}s}{r^n \log (n \log r )}\quad \text{a.s.},$$ whence ${\lim\sup}_{a\to 0+}\frac{2a \log\int_0^\infty \eee^{B(s)-as}{\rm d}s}{\log \log (1/a)}\leq 1$ a.s.\ because $r>1$ is arbitrary.

\noindent (a) Let $\mu$ be a {\it random} measure defined by $$\mu([0,\,t]):=\sum_{k=0}^{\lfloor t\rfloor} \eee^{S_k}\eta_{k+1},\quad t\geq 0$$ and take the same $\varphi$ and $\psi$ as in the proof of part (a), so that $\alpha=2$. By Theorem \ref{thm:suprLIL}, relation \eqref{eq:inter2} holds. With this at hand, the rest of the proof mimics that of part (b). Note that there is the additional factor $\sigma$ which was absent in the proof of part (a). In particular, $A=2^{1/2}\sigma$ rather than $2^{1/2}$.
\end{proof}
\begin{proof}[Proof of Corollary \ref{corr:lil}]
Under the assumptions of Theorem \ref{thm:lil}(a), invoking \eqref{limit_perturbed0} yields $$\frac{a\log\sum_{k\geq 0}\eee^{S_k-ak}\eta_{k+1}}{\log\log 1/a}~\overset{\mmp}{\longrightarrow}~0,\quad a\to 0+.$$ This in combination with the inequality $\log Y(a)\geq 0$ a.s.\ for small $a>0$, which is a consequence of $\lim_{a\to 0+} \sum_{k\geq 0}\eee^{S_k-ak}\eta_{k+1}=+\infty$ a.s., yields 
$$
\liminf_{a\to 0+}\frac{a\log \sum_{k\geq 0}\eee^{S_k-ak}\eta_{k+1}}{\log\log (1/a)}=0\quad\text{a.s.}
$$
In the setting of Theorem \ref{thm:lil}(b), relation \eqref{eq:convtoexp} entails $${\lim\inf}_{a\to 0+}\frac{a\log\int_0^\infty \eee^{B(s)-as}{\rm d}s}{\log \log 1/a}=0\quad \text{a.s.}$$ Both claims follow from the last two limit relations and Theorem \ref{thm:lil} with the help of the intermediate value theorem for continuous functions.
\end{proof}

\noindent {\bf Acknowledgement}. A. Iksanov and A. Marynych were supported by the National Research Foundation of Ukraine
(project 2020.02/0014 `Asymptotic regimes of perturbed random walks: on the edge of modern and classical probability').


\begin{thebibliography}{99}

\bibitem{Bertoin:1998} J. Bertoin, \textit{L\'{e}vy processes}. First paperback edition. Cambridge University Press, 1998.

\bibitem{Billingsley:1968} P. Billingsley, \textit{Convergence of probability measures}. Wiley, 1968.

\bibitem{Bingham+Goldie+Teugels:1989} N.~H. Bingham, C.~M. Goldie and J.~L. Teugels, \textit{Regular variation}. Cambridge University Press,
1989.

\bibitem{Bingham:1986} N.~H. Bingham, \textit{Variants on the law of the iterated logarithm}. Bull. London Math. Soc. \textbf{18} (1986), 433--467.

\bibitem{Bovier+Picco:1993} A. Bovier and P. Picco, \textit{A law of the iterated logarithm for random geometric series}. Ann. Probab. \textbf{21} (1993), 168--184.

\bibitem{Buraczewski et al:2016} D. Buraczewski, E. Damek and T. Mikosch, \textit{Stochastic models with power-law tails. The equation $X = AX + B$}. Springer, 2016.

\bibitem{Daley+Hall:1984} D.~J. Daley and P. Hall, \textit{Limit laws for the maximum of weighted and shifted i.i.d.\ random variables}. Ann. Probab. \textbf{12} (1984), 571--587.

\bibitem{Dufresne:1990} D. Dufresne, \textit{The distribution of a perpetuity, with applications to risk theory and pension funding}. Scand. Actuarial J. \textbf{1990} (1990), 39--79.

\bibitem{Durrett:2010} R, Durrett, \textit{Probability: theory and examples}. 4th Edition, Cambridge University Press,
2010.

\bibitem{Goldie+Maller:2000} C.~M. Goldie and R.~A. Maller, \textit{Stability of perpetuities}. Ann. Probab. \textbf{28} (2000), 1195--1218.

\bibitem{Iksanov:2016} A. Iksanov, \textit{Renewal theory for perturbed random walks and similar processes}. Birkh\"{a}user, 2016.

\bibitem{Iksanov+Kondratenko:2021} A. Iksanov and O. Kondratenko, \textit{Functional limit theorems for discounted exponential functional of random walk and discounted convergent perpetuity}. Statist. Probab. Letters. \textbf{176}, 109148.

\bibitem{Iksanov+Nikitin+Samoilenko:2021} A. Iksanov, A. Nikitin and I. Samoilenko, \textit{Limit theorems for discounted convergent perpetuities}. Electron. J. Probab. \textbf{26} (2021), article no. 131, 25 pp.

\bibitem{Pollak+Siegmund:1986} M. Pollak and D. Siegmund, \textit{A diffusion process and its applications to detecting a change in the drift of Brownian motion}. Biometrika. \textbf{72} (1986), 267--280.

\bibitem{Shneer+Wachtel:2011} S. Shneer and V. Wachtel, \textit{A unified approach to the heavy-traffic analysis of the maximum of random walks}. Theor Probab. Appl. \textbf{55} (2011) 332-–341.

\bibitem{Urbanik:1992} K. Urbanik, \textit{Functionals on transient stochastic processes with independent increments}. Studia Math. \textbf{103} (1992), 299--315.

\bibitem{Vervaat:1979} W. Vervaat, \textit{On a stochastic difference equation and a representation of nonnegative infinitely divisible random variables}. Adv. Appl. Probab. \textbf{11} (1979), 750--783.

\bibitem{Wang:2014} Y. Wang, \textit{Convergence to the maximum process of a fractional Brownian motion with shot noise}. Stat. Probab. Letters. \textbf{90} (2014), 33--41.

\bibitem{Yor:2001} M. Yor, \textit{Exponential functionals of Brownian motion and related processes}. Springer, 2001.

\bigskip
\end{thebibliography}
\end{document}